\documentclass[reqno]{amsart}

\usepackage{a4wide}
\usepackage{color}
\usepackage{mathrsfs}
\usepackage{mathtools}
\usepackage{amsmath}
\usepackage{amssymb}
\usepackage{nicefrac}
\usepackage{bbm}
\numberwithin{equation}{section}
\usepackage[colorlinks,citecolor=green,linkcolor=red]{hyperref}
\usepackage{url}

\usepackage[latin1]{inputenc}

\newcommand{\F}{\mathcal{F}}
\newcommand{\G}{\mathcal{G}}
\newcommand{\M}{{\sf M}}
\newcommand{\N}{\mathbb{N}}
\newcommand{\R}{\mathbb{R}}
\newcommand{\X}{{\rm X}}
\newcommand{\Y}{{\rm Y}}
\newcommand{\sfd}{{\sf d}}
\renewcommand{\d}{{\mathrm d}}
\newcommand{\1}{\mathbbm 1}
\newcommand{\eps}{\varepsilon}
\newcommand{\limi}{\varliminf}
\newcommand{\lims}{\varlimsup}
\newcommand{\LIP}{{\rm LIP}}
\newcommand{\Lip}{{\rm Lip}}
\newcommand{\BV}{{\rm BV}}
\renewcommand{\H}{{\mathcal H}}
\newcommand{\fr}{\penalty-20\null\hfill\(\blacksquare\)}

\newtheorem{theorem}{Theorem}[section]
\newtheorem{corollary}[theorem]{Corollary}
\newtheorem{lemma}[theorem]{Lemma}
\newtheorem{proposition}[theorem]{Proposition}
\newtheorem{definition}[theorem]{Definition}

\newtheorem{remark}[theorem]{Remark}

\makeatletter
\@namedef{subjclassname@2020}{\textup{2020} Mathematics Subject Classification}
\makeatother

\title[Decomposition of integral metric currents]
{Decomposition of integral metric currents}
\author{Paolo Bonicatto}
\address[P.\ Bonicatto]{Mathematics Institute, University of Warwick,
Zeeman Building, CV4 7HP Coventry, UK}
\email{Paolo.Bonicatto@warwick.ac.uk}

\author{Giacomo Del Nin}
\address[G.\ Del Nin]{Mathematics Institute, University of Warwick,
Zeeman Building, CV4 7HP Coventry, UK}
\email{Giacomo.Del-Nin@warwick.ac.uk}

\author{Enrico Pasqualetto}
\address[E.\ Pasqualetto]{Scuola Normale Superiore, Piazza dei Cavalieri 7,
56126 Pisa, Italy}
\email{enrico.pasqualetto@sns.it}

\begin{document}
\date{\today} 
\begin{abstract}
In the setting of complete metric spaces, we prove that integral
currents can be decomposed as a sum of indecomposable components.
In the special case of one-dimensional integral currents, we also show that
the indecomposable ones are exactly those associated with injective Lipschitz
curves or injective Lipschitz loops, therefore extending Federer's characterisation to metric spaces. Moreover, some applications of our
main results will be discussed.
\end{abstract}
\maketitle
\tableofcontents
\section{Introduction}
Currents have been widely used in geometric measure theory and calculus of variations as a weak setting to formulate a variety of geometric problems, especially within the theory of \textit{normal} and \textit{integral} currents developed by Federer and Fleming \cite{federer-fleming}. While the initial theory was set in Euclidean spaces, Ambrosio and Kirchheim \cite{AK} extended it to \textit{metric currents}, defined on complete metric spaces. Since then there has been a significant effort to understand which results, previously proven for Euclidean currents, could be extended to the metric setting. Among these we mention: Smirnov's result on the decomposition of Euclidean normal $1$-currents into solenoidal charges \cite{smirnov}, later extended by Paolini and Stepanov to the metric setting \cite{PS1,PS2}; isoperimetric inequalities, first proved by Federer and Fleming \cite{federer-fleming} and then by Almgren in a sharp form \cite{almgren}, extended by Wenger to metric spaces admitting a cone-type inequality \cite{Wenger1,Wenger2}.

Two useful results in the Euclidean case concern the structure of integral currents, which is particularly nice:
\begin{itemize}
    \item[\( (\rm i)\)] Every integral $k$-current $T$ admits a decomposition into countably many \textit{indecomposable} components $T_n$, with
    \[
    T=\sum_n T_n,\qquad\qquad  {\sf N}(T)=\sum_n{\sf N}(T_n),
    \]
    where ${\sf N}(T)={\sf M}(T)+{\sf M}(\partial T)$ is the normal mass;
    \item[\((\rm ii)\)] In the particular case of $1$-currents, the indecomposable ones are precisely given by (currents associated with) Lipschitz curves which are either injective or injective loops, and as a consequence every integral $1$-current is the countable sum of such curves. 
\end{itemize}
These results are stated in Federer's book \cite[4.2.25]{federer}, and a hint of a proof is provided. 
A detailed proof of these results was however missing for a long time, and it appears that only recently the proof of the structure of integral $1$-currents in the Euclidean case has been written down (see \cite[Proposition 1.3.16]{marchese} and \cite[Theorem 2.5]{CGM}), even though the characterisation of the indecomposable ones as injective Lipschitz curves or loops seems to be missing. Up to our knowledge, the proof of the decomposition in indecomposable components  was instead never explicitly written out, even in the Euclidean setting.

The aim of the present work is to rigorously prove both results in the full generality of the Ambrosio--Kirchheim setting: we will work in a complete metric space \((\X,\sfd)\), with the assumption that the cardinality of every set is an Ulam number. 
The possibility of proving both results in a complete metric space without any further geometric assumption on the space stems from the fact that an integral current is already given, and we are just looking for a suitable decomposition \textit{within} the current itself. More precisely, the strategy is as follows: we isometrically embed \(\X\) into a Banach space; here we have at our disposal the isoperimetric inequality, which is the main tool allowing us to prove the decomposition; then we just observe that the decomposition lives entirely within the current itself, and thus we can isometrically transport it back to the original metric space. We thus bypass the use of a polyhedral deformation theorem in metric spaces (whose statement would not even be clear), instead relying ultimately on the Euclidean one.

Concerning \( (\rm i)\) (see Theorem \ref{thm:decomp}) we present two proofs in Sections \ref{sec:decomp} and \ref{sec:alternative_decomp}. For the first proof we take inspiration from \cite{ACMM}, where the case of $d$-currents in $\R^d$ has been essentially proven within the theory of indecomposable finite perimeter sets. For the second proof we stay closer to the original suggestion by Federer, and prove that a greedy algorithm is successful: we inductively remove from the current an indecomposable component with maximal mass, and show that the process ends in countably many steps. Regarding this strategy we are thankful to Giada Franz, whose notes \cite{franz} were an inspiration to implement a similar strategy for currents.

Concerning \( (\rm ii)\) (see Theorem \ref{thm:opt_repr_1-curr}), an $\varepsilon$-approximation of integral $1$-currents with Lipschitz curves was already proven in length spaces by Wenger in the boundaryless case \cite{wenger_inv} and Ambrosio--Wenger in the general case \cite{AW11}. 
We take these works as a starting point and we first show that one can remove the $\varepsilon$ by means of an Arzel\`{a}--Ascoli compactness argument; then, with the aid of the Decomposition Theorem \ref{thm:decomp}, we show how it is possible to obtain injectivity, and this yields the desired description of indecomposable 1-currents.

In the final section we present some applications of the decomposition theorems. The first one, Proposition \ref{prop:ABC}, is an extension to metric spaces of a lemma by Alberti--Bianchini--Crippa \cite[Lemma 2.14]{ABC2} regarding boundaryless currents with support in a simple curve. The second one, Corollary \ref{cor:simple_planar}, is an alternative proof of the characterisation of planar simple sets (defined in \cite{ACMM}) as those bounded by Jordan curves. The proof of this result relies on Proposition \ref{prop:simple_iff}, namely the observation that a finite perimeter set of finite measure is simple if and only if the current associated with its boundary is indecomposable.

\subsubsection*{Structure of the paper} The paper is organised as follows: we use Section \ref{sec:reminder} to collect useful, previously known results on metric currents. Such section contains also a detailed explanation of our notation and should be kept as a reference, while reading the paper. Section \ref{sec:decomp} contains the first proof of the decomposition result for integral $k$-currents (Theorem \ref{thm:decomp}), based on a simple variational argument, while Section \ref{sec:alternative_decomp} is devoted to the presentation of an alternative, independent proof of the same result. Finally, in Section \ref{sec:one_indec} we present the characterisation of indecomposable 1-currents (Theorem \ref{thm:opt_repr_1-curr}) and we conclude the paper with Section \ref{sec:applications}, which contains some applications of our main results.

\subsection*{Acknowledgements} P.B.\ and G.D.N.\ have received funding from the European Research Council (ERC) under the European Union's Horizon 2020 research and innovation programme, grant agreement No 757254 (SINGULARITY). E.P.\ has been supported by the Academy of Finland (project number 314789) and by the Balzan project led by Prof.\ L.\ Ambrosio. P.B.\ wishes to thank N.A.\ Gusev for interesting discussions on currents in Euclidean spaces. G.D.N. wishes to thank G.\ Alberti for some discussions on the topic. The authors are grateful to Prof.\ S.\ Wenger for useful comments on a first draft of the paper and for suggesting the study of indecomposable $1$-currents in the metric setting.
\section{Reminder on metric currents}\label{sec:reminder}
Aim of this section is to briefly recall the theory of metric currents,
which was introduced by L.\ Ambrosio and B.\ Kirchheim in \cite{AK}.
We only discuss those definitions and results that are sufficient for
our purposes. In accordance with \cite{AK}, in this paper we will always
assume that the cardinality of any set is an Ulam number, which is consistent
with the standard ZFC set theory. This guarantees that any finite Borel measure
on a complete metric space has a separable support and is concentrated on a
\(\sigma\)-compact set.
\subsection{Main classes of metric currents}
Let \((\X,\sfd)\) be a complete metric space. Given any \(k\in\N\),
we denote by \(\mathcal D^k(\X)\) the family of all \textbf{metric
\(k\)-dimensional differential forms} on \(\X\), namely
\[
\mathcal D^k(\X)\coloneqq\LIP_b(\X)\times\LIP(\X)^k,
\]
where \(\LIP(\X)\) stands for the space of all Lipschitz, real-valued functions
on \(\X\), while \(\LIP_b(\X)\) is the set of bounded functions in \(\LIP(\X)\).
We denote by \(\Lip(f)\) the Lipschitz constant of \(f\in\LIP(\X)\).
\begin{definition}[Metric current \cite{AK}]\label{def:curr}
Let \((\X,\sfd)\) be a complete metric space. Let \(k\in\N\) be given.
Then a \textbf{metric \(k\)-current} in \(\X\) is a multilinear functional
\(T\colon\mathcal D^k(\X)\to\R\) such that:
\begin{itemize}
\item[\( (\rm i)\)] If \(f\in\LIP_b(\X)\) and \(\pi^n_i,\pi_i\in\LIP(\X)\)
satisfy \(\sup_{n\in\N}\Lip(\pi^n_i)<+\infty\) and \(\lim_n\pi^n_i(x)=\pi_i(x)\)
for every \(i=1,\ldots,k\) and \(x\in\X\), then it holds
\(T(f,\pi_1,\ldots,\pi_k)=\lim_n T(f,\pi^n_1,\ldots,\pi^n_k)\).
\item[\( (\rm ii)\)] Given any \((f,\pi_1,\ldots,\pi_k)\in\mathcal D^k(\X)\),
it holds \(T(f,\pi_1,\ldots,\pi_k)=0\) whenever there exists an index
\(i=1,\ldots,k\) such that the function \(\pi_i\) is constant on some
neighbourhood of \(\{f\neq 0\}\).
\item[\( (\rm iii)\)] There exists a finite Borel measure \(\mu\) on \(\X\)
such that
\begin{equation}\label{eq:finite_mass}
\big|T(f,\pi_1,\ldots,\pi_k)\big|\leq\prod_{i=1}^k\Lip(\pi_i)\int|f|\,\d\mu,
\quad\text{ for every }(f,\pi_1,\ldots,\pi_k)\in\mathcal D^k(\X).
\end{equation}
\end{itemize}
\end{definition}
The minimal measure \(\mu\) satisfying \eqref{eq:finite_mass} is called the
\textbf{mass measure} of \(T\) and denoted by \(\|T\|\). The \textbf{total mass}
of \(T\) is given by \({\sf M}(T)\coloneqq\|T\|(\X)\). The space of all
metric \(k\)-currents in \(\X\) is denoted by \(\mathscr M_k(\X)\).
It holds that \(\big(\mathscr M_k(\X),{\sf M}\big)\) is a Banach space.
The \textbf{support} \({\rm spt}(T)\) of a given current
\(T\in\mathscr M_k(\X)\) is defined as the support of its mass measure \(\|T\|\).
\medskip

The \textbf{boundary} of \(T\in\mathscr M_k(\X)\) is the functional
\(\partial T\colon\mathcal D^{k-1}(\X)\to\R\) defined by
\[
\partial T(f,\pi_1,\ldots,\pi_{k-1})\coloneqq T(1,f,\pi_1,\ldots,\pi_{k-1}),
\quad\text{ for every }(f,\pi_1,\ldots,\pi_{k-1})\in\mathcal D^{k-1}(\X).
\]
We say that \(T\) is \textbf{normal} provided \(\partial T\) is a metric
\((k-1)\)-current in \(\X\). We denote by \(\mathscr N_k(\X)\) the space
of all normal \(k\)-currents in \(\X\). Observe that one has
\(\partial(\partial T)=0\) for every \(T\in\mathscr N_k(\X)\). It holds that
\(\mathscr N_k(\X)\) is a Banach space if endowed with the \textbf{normal mass}
\(\sf N\), which is given by
\[
{\sf N}(T)\coloneqq{\sf M}(T)+{\sf M}(\partial T),
\quad\text{ for every }T\in\mathscr N_k(\X).
\]
Moreover, it follows from \eqref{eq:finite_mass} that any current
\(T\in\mathscr M_k(\X)\) can be uniquely extended to a multilinear
real-valued functional (still denoted by \(T\)) on the \((k+1)\)-tuples
\(L^1(\X,\|T\|)\times\LIP(\X)^k\) in such a way that the
inequality in \eqref{eq:finite_mass} is still valid.
Therefore, given any metric \(k\)-current
\(T\in\mathscr M_k(\X)\) and any Borel set \(E\subseteq\X\), we can define
the \textbf{restriction} \(T\llcorner E\in\mathscr M_k(\X)\) as
\[
T\llcorner E(f,\pi_1,\ldots,\pi_k)\coloneqq T(\1_E f,\pi_1,\ldots,\pi_k),
\quad\text{ for every }(f,\pi_1,\ldots,\pi_k)\in\mathcal D^k(\X).
\]
Another important operation is the \textbf{pushforward}:
given \((\X,\sfd_\X)\), \((\Y,\sfd_\Y)\) complete metric spaces
and \(\varphi\in\LIP(\X;\Y)\), we associate to any \(T\in\mathscr M_k(\X)\)
the current \(\varphi_\# T\in\mathscr M_k(\Y)\) that is given by
\[
\varphi_\# T(f,\pi_1,\ldots,\pi_k)\coloneqq
T(f\circ\varphi,\pi_1\circ\varphi,\ldots,\pi_k\circ\varphi),
\quad\text{ for every }(f,\pi_1,\ldots,\pi_k)\in\mathcal D^k(\Y).
\]
Note that pushforward and boundary commute, namely
\(\partial(\varphi_\# T)=\varphi_\#(\partial T)\) for all \(T\in\mathscr M_k(\X)\). The mass measure also behaves well under pushforward: \(\|\varphi_\# T\|\leq \Lip(\varphi)^k \varphi_\# \|T\|\) for all \(T\in\mathscr M_k(\X)\) and any Lipschitz map \(\varphi\in \LIP(\X;\Y)\), where $\varphi_\#\|T\|$ denotes the usual pushforward of measures.
\medskip

We say that a current \(T\in\mathscr M_k(\X)\) is \textbf{rectifiable} provided
\(\|T\|\) is concentrated on a countably \(\mathcal H^k\)-rectifiable set and
vanishes on \(\mathcal H^k\)-negligible Borel sets, where \(\mathcal H^k\)
stands for the \(k\)-dimensional Hausdorff measure on \((\X,\sfd)\). We denote
by \(\mathscr R_k(\X)\) the space of all rectifiable \(k\)-currents in \(\X\).
In addition, we say that a rectifiable current \(T\in\mathscr R_k(\X)\) is
\textbf{integer-rectifiable} provided the following property holds: given a
Lipschitz map \(\varphi\in\LIP(\X;\R^k)\) and an open set \(\Omega\subseteq\X\),
there exists a density function \(\theta\in L^1(\R^k,\mathbb Z)\) such that
\[
\varphi_\#(T\llcorner\Omega)(f,\pi_1,\ldots,\pi_k)=
\int\theta f\,{\rm det}\bigg(\frac{\partial\pi_i}{\partial x_j}\bigg)_{i,j}
\,\d\mathcal L^k,\quad\text{ for every }(f,\pi_1,\ldots,\pi_k)\in\mathcal D^k(\R^k).
\]
The existence of the partial derivatives \(\frac{\partial\pi_i}{\partial x_j}\)
is granted by the classical Rademacher's theorem. Finally, we say that
a metric \(k\)-current is \textbf{integral} provided it is both
integer-rectifiable and normal. We denote by \(\mathscr I_k(\X)\)
the space of all integral \(k\)-currents in \(\X\), which is a
\(\mathbb Z\)-module.
\medskip

We recall the following definition, which is given in
\cite[Definition 4.2.25]{federer} for Euclidean currents.
\begin{definition}[Indecomposable current]\label{def:indec_current}
Let \((\X,\sfd)\) be a complete metric space. Then a given integral current
\(T\in\mathscr I_k(\X)\) is said to be \textbf{decomposable} provided there
exists a couple of non-zero integral currents \(R,S\in\mathscr I_k(\X)\) such
that \(T=R+S\) and \({\sf N}(T)={\sf N}(R)+{\sf N}(S)\). An integral current
which is not decomposable is called \textbf{indecomposable}.
\end{definition}
\subsection{Convergence of metric currents}
There are two important notions of convergence for integral metric currents,
as we are going to recall in this section.
\medskip

Let \((\X,\sfd)\) be a complete metric space.
Let \(T\in\mathscr M_k(\X)\) and \((T_n)_{n\in\N}\subseteq\mathscr M_k(\X)\)
be given. Then we say that \(T_n\) \textbf{weakly converges} to \(T\) as
\(n\to\infty\) (briefly, \(T_n\rightharpoonup T\) as \(n\to\infty\)) provided
it holds
\[
T(f,\pi_1,\ldots,\pi_k)=\lim_{n\to\infty}T_n(f,\pi_1,\ldots,\pi_k),
\quad\text{ for every }(f,\pi_1,\ldots,\pi_k)\in\mathcal D^k(\X).
\]
It turns out that, given any open set \(\Omega\subseteq\X\), the functional
\(\mathscr M_k(\X)\ni T\mapsto\|T\|(\Omega)\in[0,\infty)\)
is lower semicontinuous when the domain is equipped with the topology
of the weak convergence. In particular, it holds that
\(\mathscr M_k(\X)\ni T\mapsto{\sf M}(T)\) is lower semicontinuous.
Moreover, the boundary operator is continuous with respect to the
weak convergence of normal currents, namely it holds
\begin{equation}\label{eq:cont_bdry}
T\in\mathscr N_k(\X),\;\;\;(T_n)_{n\in\N}\subseteq\mathscr N_k(\X),\;\;\;
T_n\rightharpoonup T\;\text{ as }n\to\infty\quad\Longrightarrow\quad
\partial T_n\rightharpoonup\partial T\;\text{ as }n\to\infty.
\end{equation}
Given a sequence \((T_n)_{n\in\N}\subseteq\mathscr M_k(\X)\), we will
denote by \(\sum_{n\in\N}T_n\in\mathscr M_k(\X)\) the weak limit of
the partial sums \(\sum_{n=1}^N T_n\) as \(N\to\infty\), whenever
such limit exists. No ambiguity should occur, as in this paper we will
just consider weakly converging series. Observe that
if \(T=\sum_{n\in\N}T_n\) exists, then for any \(\Omega\subseteq\X\) open
we have \(\|T\|(\Omega)\leq\limi_N\big\|\sum_{n=1}^N T_n\big\|(\Omega)
\leq\lim_N\sum_{n=1}^N\|T_n\|(\Omega)=\sum_{n\in\N}\|T_n\|(\Omega)\).
We thus conclude that
\begin{equation}\label{eq:mass_subadd}
\big\|{\textstyle\sum_{n\in\N}}T_n\big\|\leq{\textstyle\sum_{n\in\N}}\|T_n\|,
\quad\text{ whenever }{\textstyle\sum_{n\in\N}{\sf M}(T_n)<\infty}
\text{ (and thus }{\textstyle\sum_{n\in\N}}T_n\text{ exists),}
\end{equation}
thanks to the outer regularity of the measures \(\|T\|\)
and \(\sum_{n\in\N}\|T_n\|\).
Given that the boundary operator \(\partial\) is linear, it also holds that
\begin{equation}\label{eq:cont_bdry_series}
T\in\mathscr N_k(\X),\;\;\;(T_n)_{n\in\N}\subseteq\mathscr N_k(\X),\;\;\;
T=\sum_{n\in\N}T_n\quad\Longrightarrow\quad\partial T=\sum_{n\in\N}\partial T_n.
\end{equation}
\begin{remark}{\rm
In general, the element \(\sum_{n\in\N}T_n\) depends on the ordering
of the currents \(T_n\) in the sequence. However, this is not the case
under the additional assumption that \(C\coloneqq\sum_{n\in\N}{\sf M}(T_n)\)
is finite, since it grants that
\[
\sum_{n\in\N}\big|T_n(f,\pi_1,\ldots,\pi_k)\big|\overset{\eqref{eq:finite_mass}}
\leq C\sup_\X|f|\prod_{i=1}^k\Lip(\pi_i),\quad\text{ for every }
(f,\pi_1,\ldots,\pi_k)\in\mathcal D^k(\X).
\]
In light of this observation, we will occasionally consider series of
the form \(\sum_{n\in I}T_n\), where \(\{T_n\}_{n\in I}\) is a family
of currents that is indexed over a countable set \(I\) whose ordering
is not specified.
\fr}\end{remark}
\begin{proposition}\label{prop:alt_sum_N}
Let \((\X,\sfd)\) be a complete metric space. Let \(T\in\mathscr N_k(\X)\) and
\((T_n)_{n\in\N}\subseteq\mathscr N_k(\X)\) be such that \(T=\sum_{n\in\N}T_n\).
Then it holds that \({\sf N}(T)=\sum_{n\in\N}{\sf N}(T_n)\) if and only if
\begin{equation}\label{eq:alt_sum_N_claim}
\|T\|=\sum_{n\in\N}\|T_n\|,\qquad\|\partial T\|=\sum_{n\in\N}\|\partial T_n\|.
\end{equation}
\end{proposition}
\begin{proof}
Sufficiency is obvious. To prove necessity, assume
\({\sf N}(T)=\sum_{n\in\N}{\sf N}(T_n)\). We know from \eqref{eq:mass_subadd}
and \eqref{eq:cont_bdry_series} that \(\|T\|\leq\sum_{n\in\N}\|T_n\|\) and
\(\|\partial T\|\leq\sum_{n\in\N}\|\partial T_n\|\). For any \(E\subseteq\X\)
Borel, one has
\[\begin{split}
{\sf N}(T)&=\|T\|(\X)+\|\partial T\|(\X)=\|T\|(E)+\|T\|(\X\setminus E)
+\|\partial T\|(E)+\|\partial T\|(\X\setminus E)\\
&\leq\sum_{n\in\N}\|T_n\|(E)+\|T_n\|(\X\setminus E)+\|\partial T_n\|(E)
+\|\partial T_n\|(\X\setminus E)\\
&=\sum_{n\in\N}\|T_n\|(\X)+\|\partial T_n\|(\X)=
\sum_{n\in\N}{\sf N}(T_n)={\sf N}(T),
\end{split}\]
which forces the identities \(\|T\|(E)=\sum_{n\in\N}\|T_n\|(E)\) and
\(\|\partial T\|(E)=\sum_{n\in\N}\|\partial T_n\|(E)\). Thanks to the
arbitrariness of \(E\), we deduce that \eqref{eq:alt_sum_N_claim} is
verified, yielding the sought conclusion.
\end{proof}
On integral currents, another (more geometric) notion of convergence is
given by the \textbf{flat norm}:
\[
{\sf F}(T)\coloneqq\inf\Big\{{\sf M}(R)+{\sf M}(S)\;\Big|\;
R\in\mathscr I_k(\X),\,S\in\mathscr I_{k+1}(\X),\,T=R+\partial S\Big\},
\quad\text{ for every }T\in\mathscr I_k(\X).
\]
Observe that \({\sf F}(T)\leq{\sf M}(T)\leq{\sf N}(T)\) holds for every
\(T\in\mathscr I_k(\X)\).
\medskip

In the classical theory of currents in Euclidean spaces, a fundamental
result states that the weak convergence of currents -- when restricted to
integral currents having uniformly bounded normal mass -- is metrised
by the flat norm. This theorem has been generalised by Wenger \cite{Wenger2}
to the framework of \emph{quasi-convex}, complete metric spaces admitting
\emph{local cone-type inequalities} (a class of spaces which contains,
for instance, all Banach spaces): 
\begin{theorem}[Weak convergence and flat norm \cite{Wenger2}]
\label{thm:weak_conv_and_flat_norm}
Let \(\big(\X,\|\cdot\|\big)\) be a Banach space. Consider a
sequence \((T_n)_{n\in\N}\subseteq\mathscr I_k(\X)\) satisfying
\(\sup_n{\sf N}(T_n)<\infty\). Then for any \(T\in\mathscr I_k(\X)\)
it holds that
\[
T_n\rightharpoonup T\;\text{ as }n\to\infty\quad\Longleftrightarrow\quad
\lim_{n\to\infty}{\sf F}(T_n-T)=0.
\]
\end{theorem}
\subsection{Currentification of Lipschitz curves}
Let \((\X,\sfd)\) be a complete metric space. Then each Lipschitz curve
\(\gamma\colon[a,b]\to\X\) can be naturally associated with a metric
\(1\)-current \([\![\gamma]\!]\) as follows:
\[
[\![\gamma]\!](f,\pi)\coloneqq\int_a^b f(\gamma(t))(\pi\circ\gamma)'(t)\,\d t,
\quad\text{ for every }(f,\pi)\in\mathcal D^1(\X).
\]
We call \([\![\gamma]\!]\) the \textbf{currentification} of \(\gamma\).
It holds that \([\![\gamma]\!]\in\mathscr I_1(\X)\) and \(\big\|[\![\gamma]\!]\big\| \le \gamma_\#\big(|\dot\gamma|\,\mathcal L^1|_{[a,b]}\big)\),
where \(|\dot\gamma|\) is the metric speed of \(\gamma\). If $\gamma$ is injective, then  \(\|[\![\gamma]\!]\|= \H^{1}|_{\gamma([a,b])}\) and
\({\sf M}([\![\gamma]\!])\) coincides with the \textbf{length} of \(\gamma\), namely
\[
{\sf M}([\![\gamma]\!])=\H^1(\gamma([a,b]))=\ell(\gamma)\coloneqq\int_a^b|\dot\gamma|(t)\,\d t.
\]
Moreover, the boundary of \([\![\gamma]\!]\) is given by
\[
\partial[\![\gamma]\!]=\delta_{\gamma(b)}-\delta_{\gamma(a)},
\]
where \(\delta_x\) stands for the Dirac measure at \(x\in\X\).
In particular, it holds that \({\sf M}(\partial[\![\gamma]\!])=0\)
if and only if \(\gamma\) is a \textbf{loop}, meaning that
\(\gamma(a)=\gamma(b)\), otherwise \({\sf M}(\partial[\![\gamma]\!])=2\).
By an \textbf{injective loop} we will mean a loop \(\gamma\colon[a,b]\to\X\)
such that the restriction \(\gamma|_{[a,b)}\) is injective.
\medskip

As observed for instance in \cite{PS2}, it holds that the mapping
\begin{equation}\label{eq:continuity_currentif}
\LIP\big([0,1];\X\big)\ni\gamma\mapsto[\![\gamma]\!]\in\mathscr M_1(\X)
\quad\text{ is continuous,}
\end{equation}
where \(\LIP\big([0,1];\X\big)\) and \(\mathscr M_1(\X)\) are equipped
with the topology of the uniform convergence and the topology of the
weak convergence of currents, respectively.

In the following we will need the version of Arzel\`{a}--Ascoli's theorem stated below, that can be found in \cite{PS1}. We include for completeness a direct proof.
\begin{lemma}[Arzel\`{a}--Ascoli revisited {\cite[Proposition 2.1]{PS1}}]\label{lem:AA}
Let \((\X,\sfd)\) be a complete metric space. Consider a sequence
\((\gamma_n)_n\subseteq\LIP\big([0,1];\X\big)\) of Lipschitz
curves satisfying \(L\coloneqq\sup_n\Lip(\gamma_n)<\infty\). Suppose that for
every \(\varepsilon>0\) there exists a compact set \(K\subseteq\X\) such that
\[
\mathcal L^1\big(\big\{t\in[0,1]\,:\,\gamma_n(t)\notin K\big\}\big)
\leq\varepsilon,\quad\text{ for every }n\in\N.
\]
Then there exists a subsequence \((\gamma_{n_i})_i\) uniformly
converging to some limit curve \(\gamma\in\LIP\big([0,1];\X\big)\).
\end{lemma}
\begin{proof}
Given any integer \(m\geq 1\), let us consider the family \(\mathcal{I}^m:=(I^m_j)_{j=1}^{m}\) of subintervals of \([0,1]\) given by
\begin{equation*}
I^m_j=\left[\frac{j-1}{m},\frac{j}{m}\right],\quad \text{ for every }j=1,\ldots, m.
\end{equation*}
We choose \(\varepsilon=\frac{1}{2m}\) and consider a compact set \(K\) given by the assumptions of the Lemma. Since every subinterval \(I\in \mathcal{I}^m\) has measure \(\tfrac{1}{m}\), for every such \(I\) and for every \(n\in\N\) there exists a point \(x^I_n\in I\) such that \(\gamma_n(x^I_n)\in K\). By compactness of both \(K\) and \(I\), for every \(I\in\mathcal{I}^m\) we can ensure that, up to a not relabelled subsequence in \(n\),
\begin{align}
\begin{aligned}\label{eq:AA_cpt}
x^I_n\to x^I_\infty,&\quad\text{ for some \(x^I_\infty\in I\),}\\
\gamma_n(x^I_n)\to y^I,&\quad\text{ for some \(y^I\in K\)}
\end{aligned}
\end{align}
as \(n\to\infty\). We now claim that \(\gamma_n(x^I_\infty)\to y^I\) as \(n\to\infty\). Indeed, by triangle inequality and the uniform Lipschitz property of \(\gamma_n\) we have
\[
\sfd\big(\gamma_n(x^I_\infty),y^I\big)\leq \sfd\big(\gamma_n(x^I_\infty),\gamma_n(x^I_n)\big)+\sfd\big(\gamma_n(x^I_n),y^I\big)
\leq L|x^I_\infty-x^I_n|+\sfd\big(\gamma_n(x^I_n),y^I\big),
\]
which goes to zero by \eqref{eq:AA_cpt}. We have thus ensured that a subsequence of \(\gamma_n\) converges on all points \(x^I_\infty\), \(I\in \mathcal{I}^m\). Repeating the same procedure for every \(m\in \N\), and finally extracting a diagonal subsequence, we find \((\gamma_{n_i})_{i\in \N}\) that converges on a dense subset of \([0,1]\) as \(i\to\infty\). By the uniform Lipschitz property the convergence is uniform and holds on the entirety of \([0,1]\), thus concluding the proof.
\end{proof}
\subsection{Isoperimetric inequalities of Euclidean type}
As proved by S.\ Wenger in \cite{Wenger1}, all Banach spaces admit
an isoperimetric inequality of Euclidean type, as described in the
following statement.
\begin{theorem}[Isoperimetric inequality of Euclidean type \cite{Wenger1}]
\label{thm:isoper}
Given any \(k\in\N\), there exists a constant \(\tilde D_k>0\) such that
the following property is verified. If \(\big(\X,\|\cdot\|\big)\) is a
Banach space and a current \(T\in\mathscr I_k(\X)\) satisfies
\(\partial T=0\), then there exists \(S\in\mathscr I_{k+1}(\X)\) such
that \(\partial S=T\) and
\begin{equation}\label{eq:isoper}
{\sf M}(S)\leq\tilde D_k\,{\sf M}(T)^{\nicefrac{k+1}{k}}.
\end{equation}
\end{theorem}
In the remaining part of this paper, we will need the following
consequence of Theorem \ref{thm:isoper}.
\begin{corollary}\label{cor:conseq_isoper}
Given any \(k\in\N\), there exists a constant \(D_k>0\) such that
the following property is verified. If \(\big(\X,\|\cdot\|\big)\) is a
Banach space, then it holds that
\begin{equation}\label{eq:conseq_isoper}
{\sf F}(T)\leq D_k\,{\sf N}(T)^{\nicefrac{k+1}{k}},
\quad\text{ for every }T\in\mathscr I_k(\X).
\end{equation}
\end{corollary}
\begin{proof}
In the case where \({\sf N}(T)\geq 1\), we have that
\({\sf F}(T)\leq{\sf N}(T)\leq{\sf N}(T)^{\nicefrac{k+1}{k}}\).
Hence, let us suppose that \({\sf N}(T)<1\). If \(k=1\), then
necessarily \(\partial T=0\), whence it follows from Theorem
\ref{thm:isoper} that there exists \(S\in\mathscr I_2(\X)\) such
that \(\partial S=T\) and \({\sf M}(S)\leq\tilde D_1\,{\sf M}(T)^2\).
In particular, we obtain that
\[
{\sf F}(T)\leq{\sf M}(S)\leq\tilde D_1\,{\sf M}(T)^2\leq
\tilde D_1\,{\sf N}(T)^2.
\]
Let us now pass to the case where \(k\geq 2\). Since
\(\partial(\partial T)=0\), we know from Theorem \ref{thm:isoper}
that there exists \(R\in\mathscr I_k(\X)\) such that
\(\partial R=\partial T\) and
\({\sf M}(R)\leq\tilde D_{k-1}\,{\sf M}(\partial T)^{\nicefrac{k}{k-1}}\).
Since \(\partial(T-R)=0\), by using again Theorem \ref{thm:isoper} we
obtain \(S\in\mathscr I_{k+1}(\X)\) such that \(\partial S=T-R\)
and \({\sf M}(S)\leq\tilde D_k\,{\sf M}(T-R)^{\nicefrac{k+1}{k}}\).
Observe that \({\sf F}(T)\leq{\sf M}(R)+{\sf M}(S)\), as it follows from
\(T=R+\partial S\). Therefore, it holds that
\[\begin{split}
{\sf F}(T)&\leq{\sf M}(R)+{\sf M}(S)\leq
\tilde D_{k-1}\,{\sf M}(\partial T)^{\nicefrac{k}{k-1}}+
\tilde D_k\,{\sf M}(T-R)^{\nicefrac{k+1}{k}}\\
&\leq\tilde D_{k-1}\,{\sf M}(\partial T)^{\nicefrac{k}{k-1}}+
\tilde D_k\big({\sf M}(T)+{\sf M}(R)\big)^{\nicefrac{k+1}{k}}\\
&\leq\tilde D_{k-1}\,{\sf M}(\partial T)^{\nicefrac{k}{k-1}}
+2^{\nicefrac{1}{k}}\tilde D_k\,{\sf M}(T)^{\nicefrac{k+1}{k}}
+2^{\nicefrac{1}{k}}\tilde D_k\,{\sf M}(R)^{\nicefrac{k+1}{k}}\\
&\leq\tilde D_{k-1}\,{\sf M}(\partial T)^{\nicefrac{k}{k-1}}
+2^{\nicefrac{1}{k}}\tilde D_k\,{\sf M}(T)^{\nicefrac{k+1}{k}}
+2^{\nicefrac{1}{k}}\tilde D_k\tilde D_{k-1}^{\nicefrac{k+1}{k}}
\,{\sf M}(\partial T)^{\nicefrac{k+1}{k-1}}\\
&\leq\tilde D_{k-1}\,{\sf N}(T)^{\nicefrac{k}{k-1}}
+2^{\nicefrac{1}{k}}\tilde D_k\,{\sf N}(T)^{\nicefrac{k+1}{k}}
+2^{\nicefrac{1}{k}}\tilde D_k\tilde D_{k-1}^{\nicefrac{k+1}{k}}
\,{\sf N}(T)^{\nicefrac{k+1}{k-1}}\\
&\leq\big(\tilde D_{k-1}+2^{\nicefrac{1}{k}}\tilde D_k+
2^{\nicefrac{1}{k}}\tilde D_k\tilde D_{k-1}^{\nicefrac{k+1}{k}}\big)
{\sf N}(T)^{\nicefrac{k+1}{k}},
\end{split}\]
where in the last inequality we used the fact that \({\sf N}(T)<1\)
and \(\frac{k+1}{k}<\frac{k}{k-1}<\frac{k+1}{k-1}\). All in all,
we proved that the statement holds with
\(D_k\coloneqq 1+\tilde D_{k-1}+2^{\nicefrac{1}{k}}\tilde D_k
+2^{\nicefrac{1}{k}}\tilde D_k\tilde D_{k-1}^{\nicefrac{k+1}{k}}\),
where \(\tilde D_0\coloneqq 0\).
\end{proof}
\subsection{Behaviour of currents under isometric embeddings}
Let \((\X,\sfd_\X)\), \((\Y,\sfd_\Y)\) be complete metric spaces.
Let \(\iota\colon\X\hookrightarrow\Y\) be an isometry.
In particular, the map \(\iota\) is Lipschitz, thus to any
\(T\in\mathscr M_k(\X)\) we can associate its pushforward
\(\iota_\# T\in\mathscr M_k(\Y)\). We aim to show that currents in
\(\Y\) supported in \(\iota(\X)\) and currents in \(\X\) can be canonically
identified via the pushforward map \(\iota_\#\).
\medskip

Given any \(T\in\mathscr M_k(\X)\), it holds that
\begin{equation}\label{eq:spt_pushfrwd}
\|\iota_\# T\|=\iota_\#\|T\|,
\end{equation}
as observed for instance in the line below \cite[Eq.\ (2.4)]{AK}.
In particular, we have \({\sf M}(\iota_\# T)={\sf M}(T)\).
Moreover, since \(\iota(\X)\) is closed in \(\Y\), we deduce that
\({\rm spt}(\iota_\# T)\subseteq\iota(\X)\). Given that
\(\partial(\iota_\# T)=\iota_\#(\partial T)\), we also have that
\(\iota_\# T\in\mathscr N_k(\Y)\) whenever \(T\in\mathscr N_k(\X)\),
and in this case it holds \({\sf N}(\iota_\# T)={\sf N}(T)\).
\begin{lemma}\label{lem:iota_sharp_bij}
Let \((\X,\sfd_\X)\), \((\Y,\sfd_\Y)\) be complete metric spaces and
\(\iota\colon\X\hookrightarrow\Y\) an isometry. Fix any
\[
(\mathscr F_\X,\mathscr F_\Y)\in
\Big\{\big(\mathscr M_k(\X),\mathscr M_k(\Y)\big),
\big(\mathscr N_k(\X),\mathscr N_k(\Y)\big),
\big(\mathscr R_k(\X),\mathscr R_k(\Y)\big),
\big(\mathscr I_k(\X),\mathscr I_k(\Y)\big)\Big\}.
\]
Then the pushforward map \(\iota_\#\) is a bijection between \(\mathscr F_\X\)
and \(\big\{T'\in\mathscr F_\Y\,:\,{\rm spt}(T')\subseteq\iota(\X)\big\}\).
\end{lemma}
\begin{proof}
Let us just check that \(\iota_\#\) is a bijection from \(\mathscr M_k(\X)\)
to \(\big\{T'\in\mathscr M_k(\Y)\,:\,{\rm spt}(T')\subseteq\iota(\X)\big\}\).
We will omit the proof of the remaining claims, which can be achieved by
standard arguments. Let \(T'\in\mathscr M_k(\Y)\) be a given current satisfying
\({\rm spt}(T')\subseteq\iota(\X)\). Given any
\((f,\pi_1,\ldots,\pi_k)\in\mathcal D^k(\X)\), by virtue of
McShane's Extension Theorem we can find a \((k+1)\)-tuple
\((\tilde f,\tilde\pi_1,\ldots,\tilde\pi_k)\in\mathcal D^k(\Y)\)
such that \(\tilde f|_{\iota(\X)}=f\circ\iota^{-1}\) and
\(\tilde\pi_i|_{\iota(\X)}=\pi_i\circ\iota^{-1}\) for every \(i=1,\ldots,k\).
Then, let us define
\begin{equation}\label{eq:inv_iota_sharp}
T(f,\pi_1,\ldots,\pi_k)\coloneqq T'(\tilde f,\tilde\pi_1,\ldots,\tilde\pi_k).
\end{equation}
The resulting operator \(T\colon\mathcal D^k(\X)\to\R\) is well-defined,
because the expression appearing in the right-hand side of
\eqref{eq:inv_iota_sharp} does not depend on the specific choice of
\((\tilde f,\tilde\pi_1,\ldots,\tilde\pi_k)\). Indeed, given another \((k+1)\)-tuple
\((\tilde f',\tilde\pi'_1,\ldots,\tilde\pi'_k)\in\mathcal D^k(\Y)\)
with the same properties,  we may estimate
\[\begin{split}
&\big|T'(\tilde f,\tilde\pi_1,\ldots,\tilde\pi_k)-
T'(\tilde f',\tilde\pi'_1,\ldots,\tilde\pi'_k)\big|\\
\leq\,&\big|T'(\tilde f-\tilde f',\tilde\pi_1,\ldots,\tilde\pi_k)\big|+
\sum_{i=1}^k\big|T'(\1_{\iota(\X)}\tilde f',\tilde\pi'_1,\ldots,
\tilde\pi'_{i-1},\tilde\pi_i-\tilde\pi'_i,\tilde\pi_{i+1},\ldots,
\tilde\pi_k)\big|\\
=\,&\big|T'(\tilde f-\tilde f',\tilde\pi_1,\ldots,\tilde\pi_k)\big|\leq\prod_{i=1}^k\Lip(\tilde\pi_i)\int|\tilde f-\tilde f'|\,\d\|T'\|=
\prod_{i=1}^k\Lip(\tilde\pi_i)\int_{\iota(\X)}|\tilde f-\tilde f'|\,\d\|T'\|=0,
\end{split}\]
where we employed the locality property of metric currents;
cf.\ \cite[Theorem 3.5(iii)]{AK}. Therefore, the definition
\eqref{eq:inv_iota_sharp} of \(T\) is well-posed. Moreover, one can readily
check that \(T\) is a current on \(\X\) and that \(\iota_\# T=T'\). Finally,
the injectivity of \(\iota_\#\) can be achieved by means of similar arguments.
\end{proof}
\begin{corollary}\label{cor:decomp_embedd}
Let \((\X,\sfd_\X)\), \((\Y,\sfd_\Y)\) be complete metric spaces and let
\(\iota\colon\X\hookrightarrow\Y\) be an isometry. Fix \(T\in\mathscr I_k(\X)\)
and suppose there exists a sequence \((T'_i)_{i\in\N}\subseteq\mathscr I_k(\Y)\)
such that \(\iota_\# T=\sum_{i\in\N}T'_i\) and
\({\sf N}(\iota_\# T)=\sum_{i\in\N}{\sf N}(T'_i)\). Then there exists a
sequence \((T_i)_{i\in\N}\subseteq\mathscr I_k(\X)\) such that
\begin{equation}\label{eq:decomp_embedd_claim}
T=\sum_{i\in\N}T_i,\qquad{\sf N}(T)=\sum_{i\in\N}{\sf N}(T_i).
\end{equation}
\end{corollary}
\begin{proof}
Given any \(i\in\N\), denote by \(T_i\in\mathscr I_k(\X)\) the unique
current such that \(\iota_\# T_i=T'_i\), whose existence stems from
Lemma \ref{lem:iota_sharp_bij}. By arguing as we did in the proof of
Lemma \ref{lem:iota_sharp_bij}, we can easily show that \(T=\sum_{i\in\N}T_i\).
Indeed, given any \((f,\pi_1,\ldots,\pi_k)\in\mathcal D^k(\X)\) and
\((\tilde f,\tilde\pi_1,\ldots,\tilde\pi_k)\in\mathcal D^k(\Y)\)
such that \(\tilde f|_{\iota(\X)}=f\circ\iota^{-1}\) and
\(\tilde\pi_i|_{\iota(\X)}=\pi_i\circ\iota^{-1}\) for every \(i=1,\ldots,k\),
it holds that
\[
T(f,\pi_1,\ldots,\pi_k)=\iota_\# T(\tilde f,\tilde\pi_1,\ldots,\tilde\pi_k)
=\sum_{i\in\N}T'_i(\tilde f,\tilde\pi_1,\ldots,\tilde\pi_k)
=\sum_{i\in\N}T_i(f,\pi_1,\ldots,\pi_k).
\]
Moreover, we have that \({\sf N}(T)={\sf N}(\iota_\# T)
=\sum_{i\in\N}{\sf N}(T'_i)=\sum_{i\in\N}{\sf N}(T_i)\),
whence \eqref{eq:decomp_embedd_claim} follows.
\end{proof}
\begin{remark}\label{rmk:indecomp_embedd}{\rm
Under the same assumptions of Lemma \ref{lem:iota_sharp_bij},
for any \(T\in\mathscr I_k(\X)\) it holds that
\[
T\text{ is indecomposable }\quad\Longleftrightarrow\quad
\text{ }\iota_\# T\text{ is indecomposable.}
\]
The validity of this claim is an immediate consequence of
Lemma \ref{lem:iota_sharp_bij}.
\fr}\end{remark}
\section{Decomposition of integral currents}\label{sec:decomp}
In this section we obtain our main Decomposition Theorem \ref{thm:decomp},
which states that every integral \(k\)-current \(T\) in an arbitrary
complete metric space can be written as an (at most countable) weak sum
\(\sum_i T_i\) of \emph{indecomposable} integral \(k\)-currents \(T_i\),
having the property that \({\sf N}(T)=\sum_i{\sf N}(T_i)\).
An alternative proof of this fact will be presented in Section
\ref{sec:alternative_decomp}. Before passing to Theorem \ref{thm:decomp},
we briefly recall a classical result concerning the embeddability
of metric spaces into Banach spaces.
\begin{remark}\label{rmk:embedd_ell_infty}{\rm
Any given metric space \((\X,\sfd)\) can be isometrically embedded into
a Banach space. For instance, this can be achieved via the
\textbf{Kuratowski embedding}, that we are going to recall.
Denote by \(C_b(\X)\) the space of all bounded, continuous, real-valued
functions defined on \(\X\), which is a Banach space if endowed with the usual
pointwise operations and the supremum norm
\[
\|f\|_{C_b(\X)}\coloneqq\sup_\X|f|,\quad\text{ for every }f\in C_b(\X).
\]
Given any point \(\bar x\in\X\), we define the map
\(\iota\colon\X\hookrightarrow C_b(\X)\) as
\[
\iota(x)\coloneqq\sfd(x,\cdot)-\sfd(\bar x,\cdot),\quad\text{ for every }x\in\X.
\]
Then \(\iota\) is an isometry. Note that \(\iota\) is highly non-canonical,
as it depends on the chosen \(\bar x\).
\fr}\end{remark}
\begin{theorem}[Decomposition of integral metric currents]\label{thm:decomp}
Let \((\X,\sfd)\) be a complete metric space. Let \(T\in\mathscr I_k(\X)\) be given. Then there exists an at most countable
family \(\{T_i\}_{i\in I}\subseteq\mathscr I_k(\X)\setminus\{0\}\) of
indecomposable integral \(k\)-currents such that
\[
T=\sum_{i\in I}T_i,\qquad{\sf N}(T)=\sum_{i\in I}{\sf N}(T_i).
\]
\end{theorem}
\begin{proof}
By taking Remark \ref{rmk:embedd_ell_infty}, Corollary \ref{cor:decomp_embedd},
and Remark \ref{rmk:indecomp_embedd} into account, we can assume without loss
of generality that \(\X\) is a Banach space. Fix any constant
\(\alpha\in\big(1,\frac{k+1}{k}\big)\) and let us denote
\(\theta\coloneqq(k+1)/(\alpha k)-1>0\). We also define the family
\(\mathcal P\) as
\[
\mathcal P\coloneqq\bigg\{\{T_i\}_{i\in\N}\subseteq\mathscr I_k(\X)\;\bigg|\;
T=\sum_{i\in\N}T_i,\;{\sf N}(T)=\sum_{i\in\N}{\sf N}(T_i)\bigg\}.
\]
Notice that \(\mathcal P\) is non-empty, as it contains \(\{T,0,0,\ldots\}\).
To achieve the claim, we aim to show that
\begin{equation}\label{eq:max_pb}
M\coloneqq\sup\bigg\{\sum_{i\in\N}{\sf F}(T_i)^{\nicefrac{1}{\alpha}}
\;\bigg|\;\{T_i\}_{i\in\N}\in\mathcal P\bigg\}
\end{equation}
is finite and attained.\\
{\color{blue}\textsc{Step 1:} Compactness argument and lower semicontinuity.}
Choose any \(\big(\{T^n_i\}_{i\in\N}\big)_{n\in\N}\subseteq\mathcal P\)
such that \(\lim_n\sum_{i\in\N}{\sf F}(T^n_i)^{\nicefrac{1}{\alpha}}=M\).
Without loss of generality, we can suppose that for any \(n\in\N\) the
sequence \(\big({\sf N}(T^n_i)\big)_{i\in\N}\) is non-increasing.
Observe that the very definition of \(\mathcal P\) gives
\begin{equation}\label{eq:sum_N}
{\sf N}(T^n_i)\leq\sum_{j\in\N}{\sf N}(T^n_j)\leq{\sf N}(T),
\quad\text{ for every }i,n\in\N.
\end{equation}
The measures \(\|T\|\) and \(\|\partial T\|\) are tight (recall that
the cardinality of \(\X\) is an Ulam number by assumption). Given that
\(\|T^n_i\|\leq\|T\|\) and \(\|\partial T^n_i\|\leq\|\partial T\|\)
for every \(i,n\in\N\) by Proposition \ref{prop:alt_sum_N}, we deduce
that for any \(i\in\N\) the sequences \(\big(\|T^n_i\|\big)_n\) and
\(\big(\|\partial T^n_i\|\big)_n\) are tight, thus the compactness
properties of \(\mathscr N_k(\X)\) and the closure of \(\mathscr I_k(\X)\) (cf.\ \cite[Theorem 5.2]{AK} and \cite[Theorem 8.5]{AK}) grant
the existence of a sequence \((T_i)_{i\in\N}\subseteq\mathscr I_k(\X)\)
such that (up to a not relabelled subsequence in \(n\)) it holds that
\(T^n_i\rightharpoonup T_i\) as \(n\to\infty\) for every \(i\in\N\).
By using \eqref{eq:cont_bdry}, we deduce that
\(\partial T^n_i\rightharpoonup\partial T_i\) as \(n\to\infty\),
thus the lower semicontinuity of \(\sf M\) yields
\[
{\sf N}(T_i)={\sf M}(T_i)+{\sf M}(\partial T_i)
\leq\limi_{n\to\infty}{\sf M}(T^n_i)+\limi_{n\to\infty}{\sf M}(\partial T^n_i)
\leq\limi_{n\to\infty}{\sf N}(T^n_i),\quad\text{ for every }i\in\N.
\]
In particular, by using Fatou's lemma we obtain that
\begin{equation}\label{eq:sum_N_Ti}
\sum_{i\in\N}{\sf N}(T_i)\leq\sum_{i\in\N}\limi_{n\to\infty}{\sf N}(T^n_i)
\leq\limi_{n\to\infty}\sum_{i\in\N}{\sf N}(T^n_i)\overset{\eqref{eq:sum_N}}\leq{\sf N}(T).
\end{equation}
Moreover, by applying Theorem \ref{thm:weak_conv_and_flat_norm} we deduce that
\(\lim_n{\sf F}(T^n_i)={\sf F}(T_i)\) for every \(i\in\N\).\\
{\color{blue}\textsc{Step 2:} Uniform estimates of tails.} We claim that
\begin{equation}\label{eq:estim_tail}
\lim_{j\to\infty}\lims_{n\to\infty}\sum_{i\geq j}{\sf F}(T^n_i)=
\lim_{j\to\infty}\lims_{n\to\infty}\sum_{i\geq j}{\sf F}(T^n_i)^{\nicefrac{1}{\alpha}}=0.
\end{equation}
Given that the sequence \(i\mapsto{\sf N}(T^n_i)\) is non-increasing for any \(n\in\N\),
one has that
\begin{equation}\label{eq:N_decreas}
j\,{\sf N}(T^n_j)\leq\sum_{i\leq j}{\sf N}(T^n_i)\overset{\eqref{eq:sum_N}}\leq
{\sf N}(T),\quad\text{ for every }j,n\in\N.
\end{equation}
Thanks to \eqref{eq:N_decreas}, we can choose \(j_0\in\N\) such that
\(D_k\,{\sf N}(T^n_j)^{\nicefrac{k+1}{k}}\leq 1\)
for every \(n\in\N\) and \(j\geq j_0\), with \(D_k\) as in Corollary
\ref{cor:conseq_isoper}. Then for any \(n\in\N\) and \(j\geq j_0\)
it holds \({\sf F}(T^n_j)\leq 1\) by \eqref{eq:conseq_isoper}, whence
\[\begin{split}
\sum_{i\geq j}{\sf F}(T^n_i)&\overset{\phantom{\eqref{eq:sum_N}}}\leq
\sum_{i\geq j}{\sf F}(T^n_i)^{\nicefrac{1}{\alpha}}
\overset{\eqref{eq:conseq_isoper}}\leq
D_k^{\nicefrac{1}{\alpha}}\sum_{i\geq j}{\sf N}(T^n_i)^{\theta+1}
\overset{\eqref{eq:N_decreas}}\leq\frac{D_k^{\nicefrac{1}{\alpha}}
{\sf N}(T)^\theta}{j^\theta}\sum_{i\geq j}{\sf N}(T^n_i)\\
&\overset{\eqref{eq:sum_N}}\leq\frac{D_k^{\nicefrac{1}{\alpha}}
{\sf N}(T)^{\theta+1}}{j^\theta}\eqqcolon\lambda_j.
\end{split}\]
Consequently, we have that \(\lims_n\sum_{i\geq j}{\sf F}(T^n_i)\leq
\lims_n\sum_{i\geq j}{\sf F}(T^n_i)^{\nicefrac{1}{\alpha}}\leq\lambda_j\).
Since \(\lim_j\lambda_j=0\), we conclude that \eqref{eq:estim_tail} is verified.
As a consequence, it holds that \(\sum_{i\geq j}{\sf F}(T_i)\leq\lambda_j\)
for all \(j\geq j_0\) and
\begin{equation}\label{eq:T_i_max}
\sum_{i\in\N}{\sf F}(T_i)^{\nicefrac{1}{\alpha}}
=\lim_{n\to\infty}\sum_{i\in\N}{\sf F}(T^n_i)^{\nicefrac{1}{\alpha}}=M<\infty.
\end{equation}
{\color{blue}\textsc{Step 3:} \(\{T_i\}_{i\in I}\) belongs to \(\mathcal P\).}
We aim to prove that \(T=\sum_{i\in\N}T_i\). Fix any \(j\geq j_0\) and \(n\in\N\).
Notice that
\begin{equation*}
{\sf F}\left(T-\sum_{i\in\N}T_i\right)={\sf F}\left(\sum_{i\in\N}T^n_i-T_i\right)
\leq\sum_{i<j}{\sf F}(T^n_i-T_i)+\sum_{i\geq j}{\sf F}(T^n_i)+
\sum_{i\geq j}{\sf F}(T_i)
\end{equation*}
for every \(n\in\N\), thus accordingly
\[
{\sf F}\bigg(T-\sum_{i\in\N}T_i\bigg)
\leq\lambda_j+\lims_{n\to\infty}\sum_{i<j}{\sf F}(T^n_i-T_i)
+\lims_{n\to\infty}\sum_{i\geq j}{\sf F}(T^n_i)
\leq 2\lambda_j+\sum_{i<j}\lims_{n\to\infty}{\sf F}(T^n_i-T_i)=2\lambda_j.
\]
By letting \(j\to\infty\) we conclude that \({\sf F}\big(T-\sum_{i\in\N}T_i\big)=0\),
which means that \(T=\sum_{i\in\N}T_i\). Observe that
\({\sf N}(T)=\sum_{i\in\N}{\sf N}(T_i)\), as one inequality is granted by
\eqref{eq:mass_subadd}, while the converse one by \eqref{eq:sum_N_Ti}.
Finally, we have that \(\sum_{i\in\N}{\sf F}(T_i)^{\nicefrac{1}{\alpha}}=M\)
as a consequence of \eqref{eq:T_i_max}.\\
{\color{blue}\textsc{Step 4:} Indecomposability of \(T_i\).}
We argue by contradiction: suppose that the current \(T_j\) is decomposable for
some index \(j\in\N\). Then we can find two non-zero currents \(R,S\in\mathscr I_k(\X)\)
such that \(T_j=R+S\) and \({\sf N}(T_j)={\sf N}(R)+{\sf N}(S)\), which guarantees
that \(\{T_i\}_{i\in\N\setminus\{j\}}\cup\{R,S\}\in\mathcal P\).
On the other hand, since \(\sf F\) is subadditive and
\((0,\infty)\ni t\mapsto t^{\nicefrac{1}{\alpha}}\) is strictly concave,
we conclude that
\[
{\sf F}(R)^{\nicefrac{1}{\alpha}}+{\sf F}(S)^{\nicefrac{1}{\alpha}}
+\sum_{i\in\N\setminus\{j\}}{\sf F}(T_i)^{\nicefrac{1}{\alpha}}
>{\sf F}(R+S)^{\nicefrac{1}{\alpha}}
+\sum_{i\in\N\setminus\{j\}}{\sf F}(T_i)^{\nicefrac{1}{\alpha}}
=\sum_{i\in\N}{\sf F}(T_i)^{\nicefrac{1}{\alpha}}=M,
\]
which contradicts the maximality of \(M\). Therefore, the statement
is eventually achieved.
\end{proof}

\begin{remark}{\rm Some comments are in order:
\begin{itemize}
\item[\( (\rm i)\)] The decomposition given by Theorem \ref{thm:decomp} is in general not unique, as can be seen considering the currents associated with two perpendicular long segments intersecting at their midpoints.  Moreover, there is no maximality property for the components similar to \cite[Theorem 1]{ACMM}.
\item[\( (\rm ii)\)]
Contrarily to the case of finite perimeter sets considered in \cite{ACMM}, corresponding essentially to $d$-dimensional currents in $\R^d$ (for which the mass norm and the flat norm coincide), it is not possible to replace the flat norm in \eqref{eq:max_pb} with the mass norm. As an example, we can consider a countable sum of loops of lengths $\ell_i$, where $\sum_i \ell_i<\infty$ but $\sum_i \ell_i^\beta=\infty$ for every $\beta<1$.\fr
\end{itemize}
}\end{remark}
\section{An alternative proof of the decomposition}\label{sec:alternative_decomp}
We present a second proof of Theorem \ref{thm:decomp}, which is based on a na\"{i}ve approach: if $T$ is decomposable then we keep splitting it in two pieces, until we can not split them anymore. To make sure that such a process ends we actually ensure that at every step we remove an indecomposable component with a significant mass, which is the content of the following lemma. 
\begin{lemma}[Existence of a big indecomposable component]\label{lem:big_component}
Let \((\X,\sfd)\) be a complete metric space. Let \(T\in \mathscr{I}_k(\X)\)
be given. Then there exists \(T_1\in  \mathscr I_k(\X)\) such that \({\sf N}(T)={\sf N}(T_1)+ {\sf N}(T-T_1)\), \(T_1\) is indecomposable, and \({\sf N}(T_1)\geq D_k^{-k} \left(\tfrac{{\sf F}(T)}{{\sf N}(T)}\right)^{k}\).
\end{lemma}

\begin{proof} As in the Proof of Theorem \ref{thm:decomp}, we assume without loss of generality that $\X$ is a Banach space. 

{\color{blue}\textsc{Step 1:} Lower bound on the biggest component.} Let \(
m:=\inf\left\{ {\sf N}(T_1): (T_n)_{n} \in \G(T)\right\}
\),
where \(\G(T)\) denotes the family of good partitions of \(T\), that is those finite or countable sequences \((T_n)_{n=1}^N\) (possibly with \(N=\infty\)) such that 
\[
T=\sum_n T_n,\qquad {\sf N}(T)=\sum_n {\sf N}(T_n),
\]
and such that \({\sf N}(T_n)\) is non-increasing in $n$. We claim that \(m\geq D_k^{-k}\left(\tfrac{F(T)}{N(T)}\right)^{k}\). Indeed by Corollary \ref{cor:conseq_isoper}, for every good partition \((T_n)_n\) we have
\begin{equation}\label{eq:lower_bound_m}
{\sf F}(T)\leq \sum_n {\sf F}(T_n)\leq \sum_n D_k {\sf N}(T_n)^\frac{k+1}{k}\leq D_k {\sf N}(T_1)^\frac{1}{k}\sum_n {\sf N }(T_n)= D_k {\sf N}(T_1)^\frac{1}{k} {\sf N}(T).
\end{equation}

{\color{blue}\textsc{Step 2:} Grouping together small components.} Let \(\F(T)\) be the family of those good partitions \((T_n)_{n=1}^N\) such that \({\sf N}(T_n)\geq \tfrac{m}{2}\) for every \(n=1,\ldots, N\), except for at most one index. Then \(\F(T)\) is nonempty since \((T)\) belongs to \(\F(T)\), and 
\begin{equation}\label{eq:second_bullet}
(N-1)\frac{m}{2}\leq \sum_{n=1}^N {\sf N}(T_n)={\sf N}(T).
\end{equation}  
Therefore the number of components of a good partition in \(\F(T)\) is equibounded once we fix \(T\). 
Moreover, from every good partition in \(\mathcal{G}(T)\) we can obtain a good partition in \(\F(T)\) just grouping different currents together if their norm is below \(\tfrac{m}{2}\), and we can make it so that the biggest norm among all the elements remains the same. Indeed, given \((T_n)_n\in \G(T)\), let \(a_n:={\sf N}(T_n)\), so that \((a_n)_n\) is non-increasing. By assumption \(\sum_n a_n<\infty\). From elementary considerations we can find indices \(1\leq n_1< \ldots < n_p\leq\infty\) such that $a_n\geq m$ for every $1\leq n \leq n_1$ and 
    \begin{align*}
    & \frac{m}{2}\leq \sum_{n=n_j+1}^{n_{j+1}} a_n\leq m,\quad\text{ for every \(j=1,\ldots, p-1\)},\\
    & \sum_{n>n_p} a_n\leq \frac{m}{2}.
    \end{align*}
    Accordingly, we define 
    \begin{align*}
    & \tilde T_j:=\sum_{n=n_j+1}^{n_{j+1}} T_n,\quad\text{ for every \(j=1,\ldots, p-1\)},\\
    & \tilde T_p:=\sum_{n>n_p} T_n.
    \end{align*}
    Then \((T_n)_{n=1}^{n_1}\cup(\tilde T_j)_{j=1}^p\) defines (up to reordering) a good partition in \(\mathcal{F}(T)\).
    In particular,
    \begin{equation}\label{eq:inf_very_good}
    \inf\left\{ {\sf N}(T_1): (T_n)\in \F(T)\right\}
    \end{equation}
    also coincides with \(m\).

{\color{blue}\textsc{Step 3:} Existence of a minimum for \eqref{eq:inf_very_good}.} Let us consider a minimising sequence for \eqref{eq:inf_very_good}, indexed by \(j\in\N\), of good partitions in \(\F(T)\), denoted \((T_n^{j})_{n=1}^N\). We can assume that, without loss of generality, the cardinality of each partition is the same number \(N\). This is possible, up to extracting a subsequence, thanks to \eqref{eq:second_bullet}. By compactness of integral currents (cf.\ \cite[Theorem 5.2]{AK} and \cite[Theorem 8.5]{AK}; notice that every element of the partition is a subcurrent of \(T\), and thus the equitightness of the mass measures is satisfied) we obtain a subsequence (not relabelled) such that \(T_n^j\rightharpoonup T_n^\infty\) as \(j\to \infty\). We infer that \(T=\sum_{n=1}^N T_n^\infty\), hence
\[
{\sf N}(T)\leq \sum_{n=1}^N {\sf N}(T_n^\infty) \leq \sum_{n=1}^N \liminf_{j\to\infty} {\sf N}(T_n^j)\leq \liminf_{j\to\infty} \sum_{n=1}^N {\sf N}(T_n^j)={\sf N}(T).
\]
Therefore all inequalities are equalities, and in particular \((T_n^\infty)_{n=1}^N\) is a good partition of \(T\). By lower semicontinuity the infimum in \eqref{eq:inf_very_good} is thus achieved.

{\color{blue}\textsc{Step 4:} Finding a big indecomposable component.} We claim that at least one among the currents \(T_n^\infty\) with normal mass \(m\) is indecomposable. If not, we would obtain a good partition having a strictly lower energy than the infimum in \eqref{eq:inf_very_good}, which is impossible. This concludes the proof.
\end{proof}
\begin{proof}[Alternative proof of Theorem \ref{thm:decomp}]
We set \(S_0:=T\), and inductively define \(T_n\) and \(S_n\), \(n\geq 0\), as follows: we take \(T_n\) as an indecomposable component of \(S_n\) with maximal norm, and then we set \(S_{n+1}=S_n-T_n\). In particular by Lemma \ref{lem:big_component}
\begin{equation}\label{eq:lower_bound_N_component}
    {\sf N}(T_n)\geq D_k^{-k} \left(\frac{{\sf F}(S_n)}{{\sf N}(S_n)}\right)^{k}\geq D_k^{-k} \left(\frac{{\sf F}(S_n)}{{\sf N}(T)}\right)^{k}.
\end{equation}
Either \(S_n\) is indecomposable for some \(n\in\N\), and we stop, in which case we immediately obtain the desired decomposition; or we keep going for every \(n\in \N\). In the latter case, we observe that \({\sf F}(S_n)\to 0\) as \(n\to \infty\). 
This follows directly from \eqref{eq:lower_bound_N_component}, because \(\sum_n {\sf N}(T_n)\leq {\sf N}(T)<\infty\), and therefore \({\sf N}(T_n)\to 0\) as \(n\to\infty\).
Thus
\(T=\sum_{n=0}^\infty T_n\)
and for every \(p\in \N\)
\[
{\sf N}(T)\geq {\sf N}(S_p) +\sum_{n=0}^p {\sf N}(T_n)\geq \sum_{n=0}^p {\sf N}(T_n).
\]
Passing to the limit as \(p\to \infty\) we obtain additivity of the normal mass, and thus \(T=\sum_n T_n\) is the desired decomposition.
\end{proof}
\begin{remark}{\rm
As in the first proof of Theorem \ref{thm:decomp}, the components obtained above are in general not uniquely determined, even up to reordering.
\fr}\end{remark}
\section{Characterisation of indecomposable integral \texorpdfstring{\(1\)}{1}-currents}
\label{sec:one_indec}%
In this section, we focus our attention on integral metric \(1\)-currents.
With the aid of the Decomposition Theorem \ref{thm:decomp}, we can provide
a full characterisation of the indecomposable integral \(1\)-currents in
an arbitrary complete metric space. More specifically, we show that they
are exactly those currents induced by a Lipschitz curve which is either
injective or an injective loop.
\begin{lemma}[Indecomposable and non-cancelling implies injective]\label{lem:indec_implies_simple}
Let \((\X,\sfd)\) be a complete metric space and $\gamma:[0,1]\to \X$ a Lipschitz curve. Suppose that $[\![\gamma]\!]$ is indecomposable and that $\M([\![\gamma]\!])=\ell(\gamma)$. Then $\gamma$ is either injective or an injective loop.
\end{lemma}
\begin{proof}
Without loss of generality, we assume $\gamma$ is a constant-speed Lipschitz curve. Let us consider separately two cases. 
\begin{itemize}
	\item[\( (\rm i)\)]Suppose \(\gamma(s)=\gamma(t)\) for some \(s,t\in(0,1)\), with \(s<t\). Consider the currents
\(R\coloneqq[\![\gamma|_{[0,s]}]\!]+[\![\gamma|_{[t,1]}]\!]\)
and \(S\coloneqq[\![\gamma|_{[s,t]}]\!]\). Notice that $R,S$ are both
non-trivial, since $0<s<t<1$ and $\gamma$ is parametrized by constant-speed.
Since \([\![\gamma]\!]=R+S\) and \({\sf N}([\![\gamma]\!])={\sf N}(R)+{\sf N}(S)\),
we deduce that \([\![\gamma]\!]\) is decomposable.
\item[\( (\rm ii)\)] Suppose there exists \(t\in(0,1)\) such that
\(\gamma(t)\in\big\{\gamma(0),\gamma(1)\big\}\). Similarly to the previous case, consider the non-null
currents \(R\coloneqq[\![\gamma|_{[0,t]}]\!]\) and
\(S\coloneqq[\![\gamma|_{[t,1]}]\!]\). Since \([\![\gamma]\!]=R+S\)
and \({\sf N}([\![\gamma]\!])={\sf N}(R)+{\sf N}(S)\),
we deduce that \([\![\gamma]\!]\) is decomposable. \qedhere 
\end{itemize}
\end{proof}

Before passing to the main decomposition result of this section
(Theorem \ref{thm:opt_repr_1-curr}), let us recall its suboptimal
version obtained by Ambrosio--Wenger \cite{AW11} (see also
\cite{Wenger1} for the boundaryless case).
\begin{lemma}[Almost optimal representation of integral \(1\)-currents
\cite{AW11}]\label{lem:almost_opt_repr_1-curr}
Let \((\X,\sfd)\) be a complete, length metric space. Fix any
\(T\in\mathscr I_1(\X)\) and \(\varepsilon>0\). Then there exist finitely many
\(1\)-Lipschitz curves \(\gamma_i\colon[0,a_i]\to\X\), \(i=1,\ldots,n\),
such that \(\partial T=\sum_{i=1}^n\partial[\![\gamma_i]\!]\),
\({\sf M}(\partial T)=\sum_{i=1}^n{\sf M}(\partial[\![\gamma_i]\!])\), and
\[
{\sf M}\big(T-{\textstyle\sum_{i=1}^n}[\![\gamma_i]\!]\big)\leq
\varepsilon\,{\sf M}(T),\qquad{\textstyle\sum_{i=1}^n}a_i\leq
(1+\varepsilon)\,{\sf M}(T).
\]
\end{lemma}
By combining Theorem \ref{thm:decomp} and Lemma \ref{lem:almost_opt_repr_1-curr}
with a compactness argument (based upon the Arzel\`{a}--Ascoli-type result
stated in Lemma \ref{lem:AA}), we can obtain the following optimal
representation theorem for integral metric \(1\)-currents.
\begin{theorem}[Optimal representation of integral \(1\)-currents]
\label{thm:opt_repr_1-curr}
Let \((\X,\sfd)\) be a complete metric space. Fix any \(T\in\mathscr I_1(\X)\).
Then there exists a sequence \((\gamma_i)_i\) of injective Lipschitz curves or injective Lipschitz loops in \(\X\) such that
\[
T=\sum_{i\in\N}[\![\gamma_i]\!],\qquad{\sf N}(T)=
\sum_{i\in\N}{\sf N}([\![\gamma_i]\!]).
\]
\end{theorem}
\begin{proof}
 By taking Remark \ref{rmk:embedd_ell_infty}
and Corollary \ref{cor:decomp_embedd} into account, we can assume without
loss of generality that \(\X\) is a Banach space. We subdivide the proof into several steps:\\
{\color{blue}\textsc{Step 1:} Monotone rearrangement and bounds on the length of curves.} Fix any
\((\varepsilon_j)_{j\in\N}\subseteq(0,1)\) such that \(\varepsilon_j\searrow 0\).
For any \(j\in\N\), by using Lemma \ref{lem:almost_opt_repr_1-curr} we can
find indices \(m_j,n_j\in\N\) with \(m_j\leq n_j\) and constant-speed Lipschitz curves
\(\gamma^j_i\colon[0,1]\to\X\), \(i=1,\ldots,n_j\), such that
\(\gamma^j_i(0)\neq\gamma^j_i(1)\) for every \(i=1,\ldots,m_j\),
\(\gamma^j_i(0)=\gamma^j_i(1)\) for every \(i=m_j+1,\ldots,n_j\),
\(\partial T=\sum_{i=1}^{m_j}\partial[\![\gamma^j_i]\!]\),
\({\sf M}(\partial T)=\sum_{i=1}^{m_j}{\sf M}\big(\partial[\![\gamma^j_i]\!]\big)\), and
\begin{equation}\label{eq:choice_gamma_ij}
{\sf M}\big(T-{\textstyle\sum_{i=1}^{n_j}}[\![\gamma^j_i]\!]\big)\leq
\varepsilon_j,\qquad{\textstyle\sum_{i=1}^{n_j}}\ell(\gamma^j_i)
\leq{\sf M}(T)+\varepsilon_j.
\end{equation}
Since \({\sf M}(\partial[\![\gamma^j_i]\!])=2\) for every \(i=1,\ldots,m_j\),
we deduce that \(m_j\leq{\sf M}(\partial T)/2\), thus (up to taking a not
relabelled subsequence in \(j\)) we may assume that \(m\coloneqq m_1=m_j\)
for every \(j\in\N\). Moreover, up to relabelling the curves
\((\gamma^j_i)_{i=m+1}^{n_j}\), we may assume that
\(\big(\ell(\gamma^j_i)\big)_{i=m+1}^{n_j}\) is non-increasing.
Fix any point \(\bar x\in{\rm spt}(T)\) and set \(\gamma^j_i(t)\coloneqq\bar x\)
for every \(j\in\N\), \(i>n_j\), and \(t\in[0,1]\). Given that
\eqref{eq:choice_gamma_ij} yields \(\ell(\gamma^j_i)\leq{\sf M}(T)+1\)
for every \(i,j\in\N\), up to taking a not relabelled subsequence in \(j\) we
may assume that there exists \((\lambda_i)_{i\in\N}\subseteq[0,{\sf M}(T)+1]\)
such that \(\lim_j\ell(\gamma^j_i)=\lambda_i\) for every \(i\in\N\). Observe
that \(\lambda_i>0\) for all \(i\leq m\) and  \((\lambda_i)_{i>m}\) is
non-increasing, thus there exists a unique \(i_0\in\N\cup\{\infty\}\)
with \(i_0>m\) such that \(\lambda_i>0\) for all \(i<i_0\) and \(\lambda_i=0\)
for all \(i\geq i_0\).\\
{\color{blue}\textsc{Step 2:} Compactness of curves with non-infinitesimal length.} Given \(i<i_0\), we aim to apply Lemma
\ref{lem:AA} to the sequence \((\gamma^j_i)_j\). For any \(\varepsilon>0\),
there exist \(j_0\in\N\) and a compact set \(K\subseteq\X\)
containing \(\bar x\) such that \(\varepsilon_j\leq\varepsilon\),
\(\ell(\gamma^j_i)\geq\lambda_i/2\) for all \(j\geq j_0\),
and \(\|T\|(\X\setminus K)\leq\varepsilon\).
Using the first equation in \eqref{eq:choice_gamma_ij} and restricting to the set $K$, we obtain 
\[
\M\big(T\llcorner K-\textstyle\sum_{i'=1}^\infty [\![\gamma_{i'}^j]\!]\llcorner K\big)\leq \varepsilon_j.
\]
From the choice of the set \(K\) we also have that \(\M(T\llcorner K)\geq \M(T)-\eps\). From this, and by triangle inequality, we obtain
\begin{equation}\label{eq:mass_restr_ineq}
\M(T)-\varepsilon-\varepsilon_j\leq \M(T\llcorner K)-\varepsilon_j\leq \M\big(\textstyle\sum_{i'=1}^\infty [\![\gamma_{i'}^j]\!]\llcorner K\big)\leq \sum_{i'=1}^\infty \M\big([\![\gamma_{i'}^j]\!]\llcorner K\big).
\end{equation}
Let us define, for every curve \(\gamma_{i'}^j\), the set of bad points \(B_{i'}^j:=\big\{t\in[0,1]:\gamma_{i'}^j(t)\not\in K\big\}\).
Then \([\![\gamma_{i'}^j]\!]\llcorner K=(\gamma_{i'}^j)_\# [\![(B_{i'}^j)^c,e_1,1]\!]\) and as a consequence 
\[
\M\big([\![\gamma_{i'}^j]\!]\llcorner K\big)\leq \Lip(\gamma_{i'}^j)\big(1-\mathcal L^1(B_{i'}^j)\big)=\ell(\gamma_{i'}^j)\big(1-\mathcal L^1(B_{i'}^j)\big).
\]
Putting this together with \eqref{eq:mass_restr_ineq} and the second inequality in \eqref{eq:choice_gamma_ij}, we obtain
\[
\M(T)-2\varepsilon\leq\sum_{i'=1}^\infty \M\big([\![\gamma_{i'}^j]\!]\llcorner K\big)\leq \sum_{i'=1}^\infty \ell(\gamma_{i'}^j)\big(1-\mathcal L^1(B_{i'}^j)\big)\leq \M(T)+\varepsilon-\sum_{i'=1}^\infty \ell(\gamma_{i'}^j)\mathcal L^1(B_{i'}^j)
\]
for every $j\geq j_0$, which implies in particular that for every $i<i_0$ and every $j\geq j_0$
\[
\frac{\lambda_i}{2}\mathcal L^1(B_i^j)\leq\ell(\gamma_i^j)\mathcal L^1(B_i^j)\leq \sum_{i'=1}^\infty\ell(\gamma_{i'}^j)\mathcal L^1(B_{i'}^j)\leq 3\varepsilon,
\]
whence accordingly
\[
\mathcal L^1\big(\big\{t\in[0,1]\,:\,\gamma^j_i(t)\notin K\big\}\big)
\leq\frac{6\varepsilon}{\lambda_i},\quad\text{ for every }j\geq j_0.
\]
Also, we have that \(\gamma^j_i\) is \(\ell(\gamma^j_i)\)-Lipschitz,
thus in particular \(({\sf M}(T)+1)\)-Lipschitz, for every \(j\in\N\).
Therefore, an application of Lemma \ref{lem:AA} yields the existence of a
family \((\gamma_i)_{i<i_0}\) of Lipschitz curves \(\gamma_i\colon[0,1]\to\X\)
with the property that, up to taking a further subsequence in \(j\), it holds
\(\gamma^j_i\rightrightarrows\gamma_i\) uniformly as \(j\to\infty\) for every
\(i<i_0\). Thanks to \eqref{eq:continuity_currentif}, we infer that
\begin{equation}\label{eq:diag_conv_currentif}
[\![\gamma^j_i]\!]\rightharpoonup[\![\gamma_i]\!]\;\;\;\text{as }j\to\infty,
\quad\text{ for every }i\in\N\text{ such that }i<i_0.
\end{equation}
{\color{blue}\textsc{Step 3:} Convergence of the non-infinitesimal part.}
In the case $i_0 \in \N$, it is an immediate consequence of \eqref{eq:diag_conv_currentif} that the finite sums $\sum_{i<i_0}[\![\gamma^j_i]\!]$ converge (weakly in the sense of currents) to $\sum_{i<i_0}[\![\gamma_i]\!]$. 
We now aim to prove that the same conclusion remains valid when $i_0 = +\infty$: in this case, note first that 
\begin{equation*}
(i-m)\cdot\ell(\gamma^j_i)\leq\sum_{i'=m+1}^i
\ell(\gamma^j_{i'})\leq{\sf M}(T)+1
\end{equation*} 
for every \(i,j\in\N\) with \(m<i\).
Furthermore, since the length functional \(\ell\) is lower semicontinuous with
respect to the uniform convergence, we see that \((i-m)\cdot\ell(\gamma_i)
\leq(i-m)\limi_j\ell(\gamma^j_i)\leq{\sf M}(T)+1\) for every
\(i\in\N\) with \(m<i\). Notice also that the flat norm \(\sf F\) is
continuous along weakly converging sequences in \(\mathscr I_1(\X)\)
having bounded \(\sf N\)-norm (by Theorem \ref{thm:weak_conv_and_flat_norm}).
Consequently, we may estimate
\[\begin{split}
{\sf F}\big({\textstyle\sum_{i<i'}}[\![\gamma^j_{i'}]\!]
-[\![\gamma_{i'}]\!]\big)
&\overset{\phantom{\eqref{eq:conseq_isoper}}}\leq
{\textstyle\sum_{i<i'}}{\sf F}\big([\![\gamma^j_{i'}]\!]
-[\![\gamma_{i'}]\!]\big)\leq{\textstyle\sum_{i<i'}}
{\sf F}\big([\![\gamma^j_{i'}]\!]\big)+{\textstyle\sum_{i<i'}}
{\sf F}\big([\![\gamma_{i'}]\!]\big)\\
&\overset{\eqref{eq:conseq_isoper}}\leq
D_1{\textstyle\sum_{i<i'}}{\sf N}\big([\![\gamma^j_{i'}]\!]\big)^2
+D_1{\textstyle\sum_{i<i'}}{\sf N}\big([\![\gamma_{i'}]\!]\big)^2\\
&\overset{\phantom{\eqref{eq:conseq_isoper}}}\leq
D_1{\textstyle\sum_{i<i'}}\ell(\gamma^j_{i'})^2
+D_1{\textstyle\sum_{i<i'}}\ell(\gamma_{i'})^2\\
&\overset{\phantom{\eqref{eq:conseq_isoper}}}\leq
 D_1\frac{{\sf M}(T)+1}{i-m}{\textstyle\sum_{i<i'}}
\left(\ell(\gamma^j_{i'})+\ell(\gamma_{i'})\right)\leq 2 D_1\frac{({\sf M}(T)+1)^2}{i-m},
\end{split}\]
for every \(i,j\in\N\) with \(m<i\). This guarantees that,
given any \(\varepsilon>0\), there exists \(i\in\N\) sufficiently large such that
\(m<i\) and \({\sf F}\big(\sum_{i<i'}[\![\gamma^j_{i'}]\!]
-[\![\gamma_{i'}]\!]\big)\leq\varepsilon\) for every \(j\in\N\).
As a consequence of \eqref{eq:diag_conv_currentif} and Theorem
\ref{thm:weak_conv_and_flat_norm}, there exists \(j_0\in\N\) such that
\({\sf F}\big([\![\gamma^j_{i'}]\!]-[\![\gamma_{i'}]\!]\big)\leq\varepsilon/i\)
for every \(j\geq j_0\) and \(i'\leq i\). Hence, we have
\({\sf F}\big(\sum_{i'<i_0}[\![\gamma^j_{i'}]\!]-[\![\gamma_{i'}]\!]\big)
\leq 2\varepsilon\) for every \(j\geq j_0\), thus using again
Theorem \ref{thm:weak_conv_and_flat_norm} we get
\begin{equation}\label{eq:weak_conv_curves}
\sum_{i}[\![\gamma^j_i]\!]\rightharpoonup\sum_{i}[\![\gamma_i]\!],
\quad\text{ as }j\to\infty.
\end{equation}
{\color{blue}\textsc{Step 4:} The infinitesimal part converges to zero.} Since \(\lambda_i=0\) for all
\(i\geq i_0\) and \(\big(\ell(\gamma^j_i)\big)_{i\geq i_0}\) is
non-increasing for all \(j\in\N\), for any \(\varepsilon>0\) there exists
\(j_0\in\N\) such that \(\ell(\gamma^j_i)\leq\varepsilon\) for every
\(i\geq i_0\) and \(j\geq j_0\). Hence, Corollary \ref{cor:conseq_isoper} gives
\[\begin{split}
{\sf F}\big({\textstyle\sum_{i\geq i_0}}[\![\gamma^j_i]\!]\big)&
\leq\sum_{i\geq i_0}{\sf F}\big([\![\gamma^j_i]\!]\big)\leq
D_1\sum_{i\geq i_0}{\sf N}\big([\![\gamma^j_i]\!]\big)^2=
D_1\sum_{i\geq i_0}\ell(\gamma^j_i)^2
\leq\varepsilon D_1\sum_{i\geq i_0}\ell(\gamma^j_i)\\
&\leq\varepsilon D_1({\sf M}(T)+1),
\end{split}\]
for every \(j\geq j_0\). This implies that
\(\lim_j{\sf F}\big(\sum_{i\geq i_0}[\![\gamma^j_i]\!]\big)=0\), thus
accordingly one has \(\sum_{i\geq i_0}[\![\gamma^j_i]\!]\rightharpoonup 0\)
as \(j\to\infty\) by Theorem \ref{thm:weak_conv_and_flat_norm}.\\
{\color{blue}\textsc{Step 5:} Conclusion (decomposition).} By recalling
\eqref{eq:weak_conv_curves}, we obtain that \(\sum_{i=1}^{n_j}
[\![\gamma^j_i]\!]=\sum_{i\in\N}[\![\gamma^j_i]\!]\rightharpoonup
\sum_{i<i_0}[\![\gamma_i]\!]\) as \(j\to\infty\). The first
identity in \eqref{eq:choice_gamma_ij} yields
\({\sf M}\big(T-\sum_{i<i_0}[\![\gamma_i]\!]\big)\leq
\limi_j{\sf M}\big(T-{\textstyle\sum_{i=1}^{n_j}}
[\![\gamma^j_i]\!]\big)=0\), so that \(T=\sum_{i<i_0}[\![\gamma_i]\!]\) and
in particular \({\sf M}(T)\leq\sum_{i<i_0}{\sf M}\big([\![\gamma_i]\!]\big)\).
On the other hand, by exploiting the inequality $\M([\![\gamma_i]\!]) \le \ell(\gamma_i)$ for every $i\ \in \N$, the second identity in
\eqref{eq:choice_gamma_ij}, the uniform convergence \(\gamma^j_i
\rightrightarrows\gamma_i\), and Fatou's lemma, we get
\[
\sum_{i<i_0}{\sf M}\big([\![\gamma_i]\!]\big)\le\sum_{i<i_0}\ell(\gamma_i)
\leq\sum_{i<i_0}\limi_{j\to\infty}\ell(\gamma^j_i)\leq
\limi_{j\to\infty}\sum_{i<i_0}\ell(\gamma^j_i)\leq{\sf M}(T).
\]
All in all, we have shown that
\begin{equation}\label{eq:non_canceling_in_decomposition}
{\sf M}(T)=\sum_{i<i_0}{\sf M}\big([\![\gamma_i]\!]\big) = \sum_{i<i_0} \ell(\gamma_i).
\end{equation}
Finally, up to passing to a further subsequence in \(j\), we can additionally assume
that \(x_i\coloneqq\gamma^1_i(0)=\gamma^j_i(0)\) and
\(y_i\coloneqq\gamma^1_i(1)=\gamma^j_i(1)\) for every \(j\in\N\) and
\(i=1,\ldots,m\), which implies that \(\gamma_i(0)=x_i\) and \(\gamma_i(1)=y_i\)
for every \(i=1,\ldots,m\). In particular, one has
\({\sf M}(T)=\sum_{i=1}^m{\sf M}\big(\partial[\![\gamma_i]\!]\big)=
\sum_{i<i_0}{\sf M}\big(\partial[\![\gamma_i]\!]\big)\) and thus
\({\sf N}(T)=\sum_{i<i_0}{\sf N}\big([\![\gamma_i]\!]\big)\).\\
{\color{blue}\textsc{Step 6:} Conclusion (injectivity).} It remains to show that the Lipschitz curves $(\gamma_i)_{i}$ can be taken to be injective or injective loops. Observe that it is not restrictive at this point to assume that $[\![\gamma_i]\!]$ is indecomposable for every $i \in I$: indeed, applying Theorem \ref{thm:decomp} to each current $[\![\gamma_i]\!]$ (given by \textsc{Steps 1-5} above) we find a family of indecomposable currents $(S_i^j)_{j\in J} \subseteq \mathscr I_1(\X)$ such that  
\[
[\![\gamma_i]\!] =\sum_{j\in J}S_i^j,\qquad{\sf N}([\![\gamma_i]\!])=\sum_{j\in J}{\sf N}(S_i^j).
\]
Applying now \textsc{Steps 1-5} above to the family $(S_i^j)_{j\in J}$ and exploiting their indecomposability, we conclude that we can assume $[\![\gamma_i]\!]$ indecomposable for every $i$. On the other hand, from \eqref{eq:non_canceling_in_decomposition} it follows that for every $i \in \N$ the curve $\gamma_i$ satisfies the non-cancelling property $\M([\![\gamma_i]\!])=\ell(\gamma_i)$.
A straightforward application of Lemma \ref{lem:indec_implies_simple} yields the desired conclusion.
\end{proof}
\begin{lemma}[Injective implies indecomposable]\label{lem:simple_implies_indec}
Let \((\X,\sfd)\) be a complete metric space and let $\gamma:[0,1]\to \X$
be a Lipschitz curve, either injective or an injective loop.
Then $\M([\![\gamma]\!])=\ell(\gamma)$ and $[\![\gamma]\!]$ is indecomposable.
\end{lemma}
\begin{proof}
The non-cancelling property $\M([\![\gamma]\!])=\ell(\gamma)$ follows immediately from the injectivity of $\gamma$ on $(0,1)$, so it suffices to prove indecomposability. 

\begin{itemize}
	\item[\( (\rm i)\)] Suppose first that $\gamma$ is an injective Lipschitz curve: as a consequence of Theorem \ref{thm:opt_repr_1-curr}, to prove that
\([\![\gamma]\!]\) is indecomposable amounts to showing that if
\([\![\gamma]\!]=\sum_{i\in I}[\![\gamma_i]\!]\) and
\({\sf N}([\![\gamma]\!])=\sum_{i\in I}{\sf N}([\![\gamma_i]\!])\)
with \(\gamma_i\) simple, Lipschitz curves, then \(I\) must be a singleton.
To this aim, notice that \({\sf M}(\partial [\![\gamma]\!])=2\), thus there
exists a unique \(i\in I\) such that the curve \(\gamma_i\) is not a loop.
We have \(\partial[\![\gamma]\!]=\partial[\![\gamma_i]\!]\), so that
\(\gamma_i(0)=\gamma(0)\) and \(\gamma_i(1)=\gamma(1)\). Since
\(\|[\![\gamma_i]\!]\|\leq\|[\![\gamma]\!]\|\) by Proposition \ref{prop:alt_sum_N}, we deduce that \(\gamma_i\big([0,1]\big)\subseteq
\gamma\big([0,1]\big)\): were this false, we would find by continuity an open interval
$U \subseteq  [0,1]\setminus(\gamma_i^{-1}\circ 
\gamma)\big([0,1]\big)$. In turn, this would imply $0<\|[\![\gamma_i]\!]\|(\gamma_i(U))\leq\|[\![\gamma]\!] \|(\gamma_i(U))=0$, which is a contradiction: hence \(\gamma_i\big([0,1]\big)\subseteq
\gamma\big([0,1]\big)\). We now claim that also the reverse inclusion holds true, i.e.\ \(\gamma_i\big([0,1]\big)\supseteq
\gamma\big([0,1]\big)\). To prove this, we first note that $\gamma$ is a continuous map from a compact to a Hausdorff space, hence closed, therefore a homeomorphism between $[0,1]$ and $\gamma([0,1])$. This implies that the map $\gamma^{-1} \circ \gamma_i \colon [0,1] \to [0,1]$ is continuous; since $\gamma^{-1}(\gamma_i(0))=0$ and $\gamma^{-1}(\gamma_i(1))=1$, we conclude that $\gamma^{-1} \circ \gamma_i$ is surjective on $[0,1]$ and this yields the sought inclusion. We have thus shown that $\gamma$ and $\gamma_i$ are injective Lipschitz curves, with the same initial and final points and \(\gamma_i\big([0,1]\big)=
\gamma\big([0,1]\big)\). We can conclude that $\ell(\gamma) = \ell(\gamma_i)$ and thus, by the non-cancelling property, $\M([\![\gamma]\!])=\M([\![\gamma_i]\!])$: this implies that \(I=\{i\}\), as desired, and this concludes the proof in this case.
\item[\((\rm ii)\)] If $\gamma$ is an injective loop, we argue in a similar way: again, as a consequence of Theorem \ref{thm:opt_repr_1-curr}, to prove that
\([\![\gamma]\!]\) is indecomposable amounts to showing that if
\([\![\gamma]\!]=\sum_{i\in I}[\![\gamma_i]\!]\) and
\({\sf N}([\![\gamma]\!])=\sum_{i\in I}{\sf N}([\![\gamma_i]\!])\)
with \(\gamma_i\) simple, Lipschitz curves, then \(I\) must be a singleton.
Notice that \({\sf M}(\partial [\![\gamma]\!])=0\), thus for every \(i\in I\) the curve \(\gamma_i\) is  a loop. Fix any $i \in I$: since
\(\|[\![\gamma_i]\!]\|\leq\|[\![\gamma]\!]\|\) by Proposition \ref{prop:alt_sum_N}, we deduce that \(\gamma_i\big([0,1]\big)\subseteq
\gamma\big([0,1]\big)\). As above, a short topological argument yields also the reverse inclusion, i.e.\ \(\gamma_i\big([0,1]\big)\supseteq
\gamma\big([0,1]\big)\). Being a loop, $\gamma$ can be identified with a continuous map $\gamma \colon \mathbb S^1 \to\X$: as above, $\gamma$ is then a closed map, hence a homeomorphism between $\mathbb S^1$ and $\gamma([0,1])$. This implies that the map $\gamma^{-1} \circ \gamma_i \colon \mathbb S^1 \to \mathbb S^1$ is continuous and injective, thus (again by standard facts in topology) it must be surjective as well. We have thus shown that $\gamma$ and $\gamma_i$ are injective loops, with \(\gamma_i\big([0,1]\big)=
\gamma\big([0,1]\big)\). We can thus deduce that $\ell(\gamma) = \ell(\gamma_i)$ and the conclusion follows as in the previous case, as a consequence of the non-cancelling property.\qedhere 
\end{itemize}
\end{proof}
\begin{remark}{\rm
The following equivalence is satisfied: if \((\X,\sf d)\) is a complete metric
space and \(\gamma:[0,1]\to \X\) is a Lipschitz curve, then 
\[
\gamma\text{ is injective or an injective loop } \quad\iff\quad \M([\![\gamma]\!])=\ell(\gamma)\ \text{ and }\ [\![\gamma]\!]\text{ is indecomposable}.
\]
This can be achieved by combining Lemma \ref{lem:indec_implies_simple}
with Lemma \ref{lem:simple_implies_indec}.
\fr}\end{remark}
As an immediate consequence of Theorem \ref{thm:opt_repr_1-curr}, Lemma \ref{lem:indec_implies_simple}, and Lemma \ref{lem:simple_implies_indec},
we deduce the following characterisation of indecomposable $1$-currents:  

\begin{corollary}[Characterisation of indecomposable integral \(1\)-currents]
\label{cor:charact_indecom_one-curr}
Let \((\X,\sfd)\) be a complete metric space. Then the indecomposable integral
\(1\)-currents \(T\in\mathscr I_1(\X)\) are exactly those of the form
\(T=[\![\gamma]\!]\), where the Lipschitz curve \(\gamma\colon[0,1]\to\X\)
is either injective or an injective loop.
\end{corollary}

In particular, in the Euclidean setting, we obtain the characterisation of indecomposable $1$-currents hinted in Federer's book, injectivity included.

The careful reader might have noticed that the result we build upon to prove Theorem \ref{thm:opt_repr_1-curr} and Corollary \ref{cor:charact_indecom_one-curr}, namely \cite[Lemma 4.4]{AW11}, already makes use of Federer's claim. This opens up the possibility of a circular reasoning. However this is not the case, and we present two ways out. First we observe that Ambrosio and Wenger do not make use of the full result, namely they do not need the injectivity of the curves. Therefore one can first prove the decomposition in Euclidean spaces without the injectivity claim (as done in \cite{CGM} and \cite{marchese}) and use that result to prove Ambrosio--Wenger, and then our result. As a second way out, we can also prove the result without injectivity, prove the decomposition given by Theorem \ref{thm:decomp} (which is unrelated to Federer's injectivity claim) and then put together the two things similarly to what we did in \textsc{Step 6} of the proof of Theorem \ref{thm:opt_repr_1-curr} to obtain  injectivity.

\section{Applications}\label{sec:applications}

In this final section, we present some consequences of the decomposition and of the characterisation of indecomposable 1-currents. 

\subsection{Integral currents with support contained in a given curve}

We commence with the following generalisation of \cite[Lemma 2.14]{ABC2}. 

\begin{proposition}\label{prop:ABC} Let $(\X,\sf d)$ be a complete metric space and let $\gamma \colon [0,a] \to\X$ be a Lipschitz curve, parametrised by arc-length. Suppose that $\gamma$ is either injective or an injective loop.
Consider a $1$-dimensional integral $T \in \mathscr I_1(\X)\setminus\{0\}$, with ${\rm spt}(T) \subseteq \gamma([0,a])$. 
	\begin{itemize}
	\item[\( (\rm i)\)]If $\partial T=0$, then $\gamma$ is an injective Lipschitz loop and $T=k [\![\gamma]\!]$, for some $k \in \mathbb Z$.
	\item[\( (\rm ii)\)] If $\gamma$ is injective and $\partial T = \partial [\![\gamma]\!] = \delta_{\gamma(a)} - \delta_{\gamma(0)}$, then $T=[\![\gamma]\!]$. 
\end{itemize}
\end{proposition}

\begin{proof} Let us split the proof in some steps. 

	{\color{blue}\textsc{Step 1:} Decomposition.} Consider the current $T \ne 0$ and assume $\partial T=0$.  Applying Theorem \ref{thm:opt_repr_1-curr}, we can write 
	\begin{equation}\label{eq:dec}
	T = \sum_{i\in I}[\![\sigma_i]\!], \qquad \text{ with } \,
	\, \M(T) = \sum_{i\in I}\ell(\sigma_i) = \sum_{i\in I}\M( [\![\sigma_i]\!])
	\end{equation}
	for at most countably many injective Lipschitz loops $\sigma_i \colon \mathbb S^1 \to\X$, \(i\in I\), which we assume to be parametrised with constant speed.

	{\color{blue}\textsc{Step 2:} $\gamma$ is a loop.}
	Fix any $i \in I$ such that $\sigma_i$ is non-trivial: since ${\rm spt}(T) \subseteq C=\gamma([0,a])$ we have $\sigma_i(\mathbb S^1) \subseteq \gamma([0,a])$ (because $\Vert [\![\sigma_i]\!] \Vert \le \Vert[\![\gamma]\!] \Vert$ as measures by Proposition \ref{prop:alt_sum_N}) and this forces $\gamma$ to be a loop as well: if not, we would find a homeomorphism between $ \mathbb S^1\simeq \sigma_i(\mathbb S^1) $ (recall $\sigma_i$ is injective, hence a homeomorphism onto its image) and a closed subinterval of $[0,a]$, a contradiction.
	
	{\color{blue}\textsc{Step 3:} Homeomorphism.} 
	Since $\gamma$ is injective, it is a homeomorphism between $[0,a]^* \simeq  \mathbb S^1$ and $C$. The map $\gamma^{-1} \circ \sigma_i \colon \mathbb S^1 \to \mathbb S^1$ is thus a continuous,  injective map and hence it is also surjective. We infer that $C = \sigma_i(\mathbb S^1)$.  
	
	{\color{blue}\textsc{Step 4:} Isometries.} Let us now equip $C=\gamma([0,a])=\sigma_i(\mathbb S^1)$ with the arc-length distance  $\sfd_{al}$ (notice this depends only on the support $C$, not on the parametrisation). The map $\gamma$ is an isometric homeomorphism between $(\mathbb S^1,a\sfd_\theta/2\pi)$ and $(C, \sfd_{al})$, where $\sfd_\theta$ denotes the distance on $\mathbb S^1$. Similarly, since $\sigma_i$ is injective, it is also an isometric homeomorphism between $(\mathbb S^1,a\sfd_\theta/2\pi)$ and $(C,\sfd_{al})$. The map $\gamma^{-1} \circ \sigma_i \colon(\mathbb S^1,a\sfd_\theta/2\pi)\to(\mathbb S^1,a\sfd_\theta/2\pi)$ is thus an isometry and hence
	$[\![\gamma]\!]= [\![\gamma \circ ( \gamma^{-1} \circ \sigma_i ) ]\!]= [\![\sigma_i]\!]$. 
		
	{\color{blue}\textsc{Step 5:} Conclusion in the boundaryless case.} From \eqref{eq:dec}, we deduce that
	\begin{equation*}
	\M(T) = \sum_{i\in I}\M([\![\sigma_i]\!]) = \# I\cdot\M([\![\gamma]\!]),
	\end{equation*}
	hence the set of indices \(I\) is finite and this concludes the proof in the case $\partial T=0$.
	
	{\color{blue}\textsc{Step 6:} The case with boundary.} The proof of Point \( (\rm ii)\) can be done exploiting a similar strategy. First, one decomposes $T$ as  
	\begin{equation}\label{eq:additivity_of_mass_step6} 
	T = \sum_{i=1}^{k}  [\![\beta_i]\!] + \sum_{i \in I}  [\![\sigma_i]\!] , \qquad \text{ with } \,\, \M(T) = \sum_{i=1}^{k} \M( [\![\beta_i]\!]) + \sum_{i \in I} \M( [\![\sigma_i]\!]),
	\end{equation}
	where $\beta_i \colon [0,1] \to\X$ are finitely many injective Lipschitz curves, parametrised with constant speed, while $\sigma_i$ are at most countably many injective Lipschitz loops. It is readily seen that ${\rm spt}(T) \subseteq \gamma([0,a])$ implies that 
	the loops $\sigma_i$ are all trivial (proceed as in \textsc{Step 2}: the map $\gamma^{-1} \circ \sigma_i$ would be a homeomorphism between $\mathbb S^1$ and a closed subinterval of $[0,a]$); as for the curves $\beta_i$, from $\partial T = \partial [\![\gamma]\!]$ we infer that there exists one and only one $i_0\in \{1, \ldots, k\}$ such that $\beta_{i_0}$ is non-trivial. Such $\beta_{i_0}$ is thus an injective, constant-speed Lipschitz curve whose image is contained in $\gamma([0,a])$ and $\beta_{i_0}(1)=\gamma(a)$ and $\beta_{i_0}(0)=\gamma(0)$, hence (arguing similarly as above) we conclude $\beta_{i_0}(t)= \gamma(at)$ for all \(t\in[0,1]\), in particular $[\![\beta_{i_0}]\!]=[\![\gamma]\!]$ and the proof is completed also in this case.
\end{proof}

\subsection{On the structure of simple sets in \texorpdfstring{$\R^2$}{R2}}

In this paragraph, we show how it is possible to derive from the decomposition some consequences on the structure of particular sets of finite perimeter in $\R^2$. The material presented here should be compared with \cite{ACMM}, to which we refer the reader for the notation and for a more detailed analysis. 
\medskip

For any measurable set $A \subseteq \R^d$ we denote by $T_A$ the canonical current associated with $A$ (with unit density and orientation induced by $\R^d$); 
observe that if $\mathcal{L}^d(A)<+\infty$ then $T_A \in \mathscr R_d(\R^d)$ and if $A$, in addition, has also finite perimeter, then $T_A \in \mathscr I_d(\R^d)$. In this case, the current $\partial T_A$ is the natural current associated with the reduced boundary $\partial^*A$ with orientation compatible with Stokes' Theorem (we refer the reader to \cite{AFP} for a comprehensive treatment of the theory of sets of finite perimeter).
A set $A$ of finite perimeter is said to be \textbf{decomposable} if there exist two measurable sets $B,C$ with $A=B \cup C$, with $\mathcal L^d(B \cap C)=0$,  $\mathcal L^d(B)>0$, $\mathcal L^d(C)>0$, and  
$\text{Per}(A) = \text{Per}(B) +\text{Per}(C)$. A set $A$ of finite perimeter is said to be \textbf{indecomposable} if it is not decomposable. It is easily seen that $A$ is indecomposable if and only if $T_A$ is indecomposable.
Further, we recall the following notion: 

\begin{definition}[{\cite[Definition 3]{ACMM} and \cite[Proposition 2.17]{BG}}]
	A set $A\subseteq \R^d$ of finite perimeter is said to be \textbf{simple} if
	\begin{itemize}
		\item[\( (\rm i)\)] either $A=\R^d$;
		\item[\( (\rm ii)\)] or $\mathcal L^d(A)<+\infty$ and both $A,A^c$ are indecomposable.
	\end{itemize}
\end{definition}

We look for a criterion for simple sets in terms of the decomposability of its associated current. Before stating our result, we present the following observation. 

\begin{remark}[Indecomposability of a current and of its boundary]\label{rem:_indec_boundary_current} {\rm If $T \in \mathscr I_d(\R^d)$, the indecomposability of $\partial T $ implies the indecomposability of $T$. Indeed, if $T=U+V$, with $\mathsf N(T)=\mathsf N(U)+\mathsf N(V)$, then we would have $\partial T = \partial U + \partial V$, with $\mathsf N(\partial T)=\mathsf N(\partial U)+\mathsf N(\partial V)$ and this necessarily implies $\partial U = 0$ or $\partial V = 0$. By the Constancy Lemma and finiteness of the mass, either $U = 0$ or $V=0$, whence the indecomposability of $T$ follows. 

\noindent Observe that, except for the particular case mentioned above, there is no other implication between the indecomposability of $T$ and the indecomposability of $\partial T$. In one direction, it is enough to consider the $1$-integral current associated with two loops, which is decomposable but whose boundary is zero (hence indecomposable). Conversely, the $2$-current associated with an annulus in $\R^2$ is indecomposable but its boundary is not. \fr}\end{remark}

We are finally ready to state and prove the following criterion: 

\begin{proposition}\label{prop:simple_iff}
	Let $E \subseteq \R^d$ be a set of finite perimeter and of finite Lebesgue measure. The set $E$ is simple if, and only if, the current $\partial T_E$ is indecomposable.
\end{proposition} 

\begin{proof}
Assume $\partial T_E$ is indecomposable. By Remark \ref{rem:_indec_boundary_current}, the current $T_E$ is indecomposable, hence $E$ is indecomposable. It remains to show that $E^c$ is indecomposable. Let us consider two sets $A,B$ such that $E^c = A \cup B$ with $\text{Per}(E^c) = \text{Per}(A) +\text{Per}(B)$ and $\mathcal L^d(A \cap B)=0$. In view of \cite[Remark 1]{ACMM}, we can assume that $\mathcal L^d(A)<\infty$ and $\mathcal L^d(B)=+\infty$. It can be readily checked that $B^c = E \cup A$ with $\mathcal L^d(E \cap A)=0$, so we can consider the currents $T_{A}, T_{B^c}$ which satisfy $T_{B^c} = T_E + T_A$, whence $\partial T_{E} = \partial T_{B^c} - \partial T_A$.  Since $\text{Per}(E) = \text{Per}(E^c) = \text{Per}(A) +\text{Per}(B) =\text{Per}(A) +\text{Per}(B^c)$, by indecomposability of $\partial T_E$ we deduce either $\partial T_{B^c}=0$ or $\partial T_A=0$ and both cases force $T_A=0$, hence $\mathcal L^d(A)=0$ and the proof is complete. 

Conversely, let us suppose that $E$ is simple and let $\partial T_E = U + V$, for some $U,V \in \mathscr I_{d-1}(\R^d)$ with $\partial U,\partial V=0$ and with 
\begin{equation}\label{eq:additivity}
\M(\partial T_E) = \M(U) + \M(V). 
\end{equation}
By the Isoperimetric Inequality (see Theorem \ref{thm:isoper}) and the Constancy Lemma, one can uniquely determine $X,Y \in \mathscr I_d(\R^d)$ \emph{of finite mass} such that $\partial X = U$ and $\partial Y = V$. By standard facts, $X =  [\![\R^d, e_1 \wedge \ldots \wedge e_d, \vartheta]\!]$ and $Y =  [\![\R^d, e_1 \wedge \ldots \wedge e_d, \psi]\!]$ for some functions $\vartheta, \psi \in \BV(\R^d;\mathbb Z)$. The goal is to show that $D\vartheta = 0$ or $D\psi = 0$. In view of \eqref{eq:additivity}, we can write 
\begin{equation}\label{eq:additivity_as_measures}
\H^{d-1}|_{\partial^*E} = |D\vartheta| + |D\psi|
\end{equation}
as measures on $\R^d$; testing \eqref{eq:additivity_as_measures} on the set $E^{(1)}$ (the set of density points of $E$) we deduce  $|D\vartheta|(E^{(1)}) + |D\psi|(E^{(1)})=0$ and hence, in view of \cite[Remark 2]{ACMM} and of the indecomposability of $E$, we infer $\vartheta$ and $\psi$ are constant in $E$. Arguing similarly on $(E^c)^{(1)}$, we deduce that $\vartheta$ and $\psi$ vanish in $E^c$ (recall that $\mathcal L^d(E)<\infty$ and $\vartheta, \psi \in L^1(\R^d)$). The conclusion is thus $\vartheta = \alpha \1_{E}$ and $\psi = \beta \1_{E}$, for some $\alpha,\beta \in \mathbb Z$. By \eqref{eq:additivity} we must have $|\alpha| + |\beta| = 1$, which forces either $\alpha =0$ or $\beta =0$, which readily implies either $D\vartheta = 0$ or $D\psi = 0$, and the proof is thus completed.
\end{proof}
As a consequence, we deduce the following result about the structure of simple sets in $\R^2$:
\begin{corollary}[{\cite[Theorem 7]{ACMM}}]\label{cor:simple_planar} Let $E \subseteq \R^2$ be a set of finite perimeter and finite measure. Then $E$ is simple if, and only if, its reduced boundary is equal (up to an $\H^1$-negligible subset) to the image of an injective Lipschitz loop.\end{corollary}
\end{document}